\documentclass[a4paper,12pt]{article}
\usepackage{fullpage}
\usepackage[pdftex]{graphicx}

\usepackage{titling}

\usepackage{amscd}
\usepackage{amsmath,amsthm, amssymb}
\usepackage{enumerate}
\usepackage{mathtools}
\usepackage{color}
\usepackage[all]{xy}

\newcommand{\cI}{\mathcal{I}}
\newcommand{\cQ}{\mathcal{Q}}

\newcommand{\Z}{\mathbb{Z}}
\newcommand{\K}{\mathbb{K}}
\newcommand{\Gl}{\mathrm{GL}}
\newcommand{\ledom}{\trianglelefteq}

\newcommand{\End}{\mathrm{End}}
\DeclareMathOperator{\ch}{char}

\newcommand{\A}{\tau}
\newcommand{\Hom}{\mathrm{Hom}}
\newcommand{\Coker}{\mathrm{Coker}}
\newcommand{\Ker}{\mathrm{Ker}}
\newcommand{\im}{\mathrm{Im}}

\newcommand{\stiy}[1]{}
\newcommand{\stiyy}[1]{}

\newcommand{\stiiy}[1]{}

\newcommand{\stiiyy}[1]{}

\newtheorem{theorem}{Theorem}[section]
\newtheorem{lemma}[theorem]{Lemma}
\newtheorem{corollary}[theorem]{Corollary}
\newtheorem{proposition}[theorem]{Proposition}

\newtheorem{observation}[theorem]{Observation}
\theoremstyle{definition}
\newtheorem{definition}[theorem]{Definition}

\theoremstyle{remark}
\newtheorem{remark}[theorem]{Remark}

\newtheorem{example}[theorem]{Example}

\numberwithin{equation}{section}

\newcommand{\si}[1]{\small{\emph{#1}}}

\preauthor{}
\postauthor{}
\DeclareRobustCommand{\authorthing}{
\begin{center}
	\begin{tabular}{lll}
		Karin Erdmann & Ana Paula Santana\thanks{This work was partially supported by the Centro de Matem\'atica da
Universidade de Coimbra (CMUC), funded by the European Regional
Development Fund through the program COMPETE and by the Portuguese
Government through the FCT - Funda\c{c}\~ao para a Ci\^encia e a Tecnologia
under the project PEst-C/MAT/UI0324/2013.
} &
		Ivan Yudin\thanksmark{1}\thanksgap{0.3em}\thanks{The third
		author's work was partially supported by the FCT Grant
SFRH/BPD/31788/2006.}\\	
\si{Mathematical}
& \si{CMUC, Department of} & \si{CMUC, Department of}\\
\si{Institute}&
\si{Mathematics}&
\si{Mathematics}\\
\si{University of Oxford}
&\si{University of Coimbra}
&\si{University of Coimbra}\\
\si{Oxford}
&\si{Coimbra} & \si{Coimbra}\\
\si{UK} & \si{Portugal} &\si{Portugal}\\
\si{erdmann@maths.ox.ac.uk} & \si{aps@mat.uc.pt} & \si{yudin@mat.uc.pt}
	\end{tabular}
\end{center}
}
\date{}
\author{\authorthing}

\title{On Auslander-Reiten sequences for Borel-Schur algebras}

\begin{document}
\maketitle
\abstract{We classify Borel-Schur algebras having finite representation type. We also determine Auslander-Reiten sequences for a large class
of simple modules over Borel-Schur algebras. A partial information on the structure
of the socles of Borel-Schur algebras is given.
}

\section{Introduction}
 Consider the general linear group $GL_n(\K)$, where $\K$ is an infinite field,
and let $B^+$ be the Borel subgroup of $GL_n(\K)$ consisting of all upper
triangular matrices in $GL_n\left( \K \right)$.
The Schur algebras $S(n,r)$ and $S(B^+) :=S(B^+,n,r)$ corresponding to
$GL_n(\K)$ and $B^+$, respectively, are powerful tools in the study
of polynomial representations of $GL_n(\K)$ and $B^+$.
In particular, the
simple modules of $S(B^+)$ labelled by partitions induce to Weyl modules
for $S(n,r)$, and Weyl modules are central objects of study.
In the recent paper~\cite{apsiy}, Borel-Schur algebras were crucial to construct
resolutions for Weyl modules.   Therefore one would like to  understand
better the algebra  $S(B^+)$.

Given a finite dimensional algebra over $\K$, we denote by $A\mbox{-mod}$ the category of all finite dimensional left $A$-modules.
The algebra $A$ is said to have finite representation type if there are only finitely many isomorphism classes of indecomposable modules in $A\mbox{-mod}$.
Representation type of Schur algebras and of infinitesimal Schur algebras
was determined in \cite{K1} and \cite{reptype}.
In particular, it is known when Schur algebras are of finite type. In this paper we obtain the corresponding classification for Borel-Schur algebras, determining the conditions on $n$, $r$  and characteristic of $\K$
under which $S(B^+,n,r)$ is of finite representation type. This is Theorem~\ref{reptype}.

One of the motivations for this classification was  the  construction of  Auslander-Reiten sequences for Borel-Schur algebras.
The class of these sequences, also known as almost split sequences, is an important invariant of the module
category of a finite-dimensional algebra. It provides part of a
presentation of the module category.

Taking advantage of the easy multiplication of some basis elements of $S(B^+),$ we determine Auslander-Reiten sequences for a large class
of simple $S(B^+)$-modules.
 We are able to do this, for an arbitrary
$n$, under some combinatorial conditions.
We note that
when these are satisfied, the relevant simple module does not occur
in the socle of $S(B^+)$.

Several recipes were given in  the 80's for the construction of
Auslander-Reiten sequences. Although we do not use any of these, we should remark that the recipe
due to J.A.~Green~\cite{greenunpublished} was the motivation for this work.

The paper is organized as follows.
Section~\ref{notation} recalls the definitions of the algebras, and
some basic background.
In Section~\ref{almostsplit}, we construct Auslander-Reiten sequences
ending in a simple module $\K_{\lambda}$, where
$\lambda$ satisfies a condition given in~\eqref{cond}. As a by-product
we see that this condition imples that $\K_{\lambda}$ does
not occur in the (left) socle of the algebra.
The main result of this section is Theorem~\ref{prop1}.

Some observations about  the middle term of the Auslander-Reiten sequences  are given in Section~\ref{newsection}.

In Section~\ref{n2} we consider $n=2$ and find  Auslander-Reiten sequences ending in
an arbitrary simple module, that is we deal with the cases missing
in Section~\ref{almostsplit}. As an easy consequence of the results in
this section we can obtain a necessary and sufficient condition for a simple
module to occur in the socle of $S(B^+,2,r)$.
This and other results involving the socle of the Borel-Schur algebra are
summarized in Section~\ref{soclesection}. 

Section~\ref{n3} considers some
cases of Auslander-Reiten sequences, not covered in Section~\ref{almostsplit},
for the algebra $S(B^+,3,r)$. 
%

In Section~\ref{functor} we discuss reduction of rank. This may be of more general
interest, and is in fact used in Section~\ref{finitetype}, where we determine precisely which Borel-Schur algebras are of finite type.

\section{Notation and basic results}
\label{notation}
In this section we establish the notation we will use and give some basic
results. We will follow \cite{aps} and any undefined term may
be found there. 
For further details on the general theory of Schur algebras see \cite{green} and
\cite{martin_book}.  

Throughout the paper $\K$ is an infinite field of arbitrary characteristic, $n$ and $r$ are arbitrary
fixed positive integers and $p$ is any prime number.

For any natural number $s$, we denote by $\mathbf{s}$ the set $\left\{ 1,\dots,
s
\right\}$ and by $\Sigma_s$ the symmetric group on $\mathbf{s}$.
Define the sets of multi-indices $I\left( n,r \right)$ and of
compositions
$\Lambda(n,r)$ by
\begin{align*}
 I(n,r)& =
\left\{ i=\left( i_1,\dots, i_r \right) \, \middle|\, i_\rho \in
\mathbf{n} \mbox{ all } \rho\in \mathbf{r} \right\}\\
\Lambda(n,r)& = \{\, \lambda = (\lambda_1,\dots, \lambda_n) \,|\,
\lambda_\nu\in \Z,\ \lambda_\nu\ge 0\ (\nu\in \mathbf{n}),\ \sum_{\nu\in
\mathbf{n}} \lambda_\nu = r \}.
\end{align*}
We will often write $I$ instead of $I(n,r)$ and $\Lambda$ instead of
$\Lambda(n,r)$.

Given $i\in I$ and  $\lambda\in \Lambda$, we say that $i$ has \emph{weight
$\lambda$} and
write $i\in \lambda$ if
$\lambda_\nu = \# \left\{\,  \rho\in \mathbf{r} \,\middle|\,  i_\rho = \nu
\right\}$, for $\nu \in \mathbf{n}$.

The group $\Sigma_r$ acts on the right of $I$ and of $I\times I$, respectively, by
$i\pi = \left( i_{\pi 1}, \dots, i_{\pi r} \right)$ and $(i,j)\pi = \left( i\pi,
j\pi
\right)$, all $\pi \in \Sigma_r$ and $i$, $j\in I$.
If $i$ and $j$ are in the same $\Sigma_r$-orbit of $I$ we write $i\sim j$.
Also $\left( i,j \right)\sim \left( i',j' \right)$ means these two pairs are in
the same $\Sigma_r$-orbit of $I\times I$.
We denote the stabilizer of $i$ in $\Sigma_r$ by $\Sigma_i$,
that is   $\Sigma_i =
\left\{\, \pi\in \Sigma_r \,\middle|\, i\pi = i \right\}$. We write
$\Sigma_{i,j} = \Sigma_i \cap \Sigma_j$. Given $i$, $j\in I$, then
 $i\le j$ means that
$i_\rho\le j_\rho$ for all $\rho\in \mathbf{r}$, and $i<j$ means that $i\le j$ and
$i\not=j$.

We use $\ledom$ for the ``dominance order'' on $\Lambda$, that is
$\alpha\ledom\beta$ if
$\sum_{\nu=1}^\mu \alpha_\nu \le \sum_{\nu=1}^\mu \beta_\nu$ for all $\mu\in
\mathbf{n}$. Obviously if $i\in\alpha$ and $j\in\beta$ (where $\alpha$, $\beta\in
\Lambda$), then $i\le j$ implies~$\beta\ledom \alpha$.

Given $\lambda\in\Lambda$, we consider in $I$ the special element
$$
l = l(\lambda) = ( \underbrace{1,\dots,1}_{\lambda_1},
\underbrace{2,\dots,2}_{\lambda_2}, \dots, \underbrace{n,\dots, n}_{\lambda_n}
).
$$
Clearly $\Sigma_{l(\lambda)}$ is the parabolic subgroup
associated with $\lambda$
\begin{equation*}
\Sigma_\lambda =
\Sigma_{\{1,\dots,\lambda_1\}}\times \Sigma_{\{\lambda_1 + 1, \dots, \lambda_1+\lambda_2\}}\times
\dots \times \Sigma_{\{\lambda_1 + \dots + \lambda_{n-1} + 1,\dots, r\}}.
\end{equation*}
	For each $\nu \in \mathbf{n-1}$, and each non-negative integer $m\le
\lambda_{\nu+1}$, we define
\begin{equation*}
\lambda(\nu,m) = \left( \lambda_1, \dots,
\lambda_{\nu} + m, \lambda_{\nu+1} -m, \dots, \lambda_n \right)\in \Lambda,
\end{equation*}
	and
write $l(\nu,m)$ for $l\left( \lambda\left( \nu,m \right) \right)$. We have
$l(\nu,m)\le l$.

For the notation of $\lambda$-tableaux the reader is referred to \cite{aps}.
Given $\lambda= \left( \lambda_1,\dots, \lambda_n \right)\in\Lambda$, we choose
the basic $\lambda$-tableau
\begin{equation*}
	T^\lambda =
	\begin{array}{cccccc}
		1 & 2 & \dots & \lambda_1\\
		\lambda_1 + 1 & \lambda_1 + 2 & \dots & \dots & \dots & \lambda_1 +
		\lambda_2 \\
		 \dots \\
		 \lambda_1 + \dots + \lambda_{n-1} + 1 & \dots & \dots & \dots & r
	\end{array}
\end{equation*}
The row-stabilizer of $T^\lambda$, i.e. the subgroup of $\Sigma_r$ consisting of
all those $\pi\in \Sigma_r$ which preserve the rows of $T^\lambda$ is the
parabolic subgroup $\Sigma_\lambda$.

Given $i\in I$, we define the $\lambda$-tableau $T_i^\lambda$ as
\begin{equation*}
	T_i^\lambda =
	\begin{array}{cccccc}
		i_1 & i_2 & \dots & i_{\lambda_1}\\
		i_{\lambda_1 + 1} & i_{\lambda_1+2} & \dots & \dots & \dots &
		i_{\lambda_1 + \lambda_2}\\
		\dots \\
		i_{\lambda_1 +\dots \lambda_{n-1} + 1} & \dots & \dots & \dots &
		i_r.
	\end{array}
\end{equation*}
Then $T_{l}^\lambda$ has only $1$'s in the first row,
$2$'s in the second row, \dots, $n$'s in row $n$. Notice also that
$T^\lambda_{l(\nu,m)}$ differs from $T^\lambda_{l}$ only by
the first $m$ entries of row  $\nu+1$: these entries are all equal to $\nu$.

We say that a $\lambda$-tableau $T^\lambda_i$ is row-semistandard if the entries
in each row of $T^\lambda_i$ are weakly increasing from left to right. We define
\begin{equation*}
	I\left( \lambda \right) := \left\{\,  i\in I \,\middle|\, i\le l\left(
	\lambda \right) \mbox{ and  } T^\lambda_i \mbox{ is row-semistandard }
	\right\}
\end{equation*}
and
\begin{equation*}
	J\left( \lambda \right) := \left\{\, j\in I \,\middle|\,  j\ge l\left(
	\lambda \right)\mbox{ and } T^\lambda_j\mbox{ is row-semistandard}
	\right\}.
\end{equation*}
The following obvious fact will be used later in this paper:
\begin{equation}
	\label{eq01}
	\mbox{If $\lambda_n\not=0$ and $m\le \lambda_n$, then $J\left( \lambda
	\right) = \left\{\, j\in J\left( \lambda(n-1,m) \right) \,\middle|\,
	j\ge l\left( \lambda \right) \right\}$. }
\end{equation}
Next we recall the definition of Schur algebra and of  Borel-Schur algebra
as they were introduced in \cite{green1}.

The general linear group $\Gl_n(\K)$ acts on $\K^n$ by
multiplication. So $\Gl_n\left( \K \right)$ acts on the $r$-fold tensor product
$\left( \K^n \right)^{\otimes r}$ by the rule
\begin{equation*}
	g\left( v_1\otimes \dots \otimes v_r \right) = gv_1 \otimes \dots
	\otimes gv_r, \mbox{ all $g\in \Gl_n\left( \K \right)$, $v_1$, \dots,
	$v_r\in \K^n $}.
\end{equation*}
Extending by linearity this action to the group algebra $\K\Gl_n\left( \K
\right)$, we obtain a homomorphism of algebras
$	T\colon \K\Gl_n\left( \K \right) \to \End_\K\left( \left( \K^n
	\right)^{\otimes r} \right).$
The image of $T$, i.e. $T\left( \K\Gl_n\left( \K \right) \right)$ is called
the \emph{Schur algebra} for $\K$, $n$, $r$ and is denoted by $S\left( n,r
\right)$. Let $B^+ = B_\K^+\left( n,r \right)$ denote the Borel subgroup of
$\Gl_n\left( \K \right)$ consisting of all upper triangular matrices in
$\Gl_n\left(\K  \right)$. The \emph{Borel-Schur algebra} $S\left( B^+
\right) = S\left( B^+, n,r \right)$ is the subalgebra $T\left( \K B^+ \right)$ of
$S\left(n,r \right)$.

Associated with each pair $\left( i,j \right)\in I\times I$, there is a well
defined element $\xi_{i,j}$ of $S(n,r)$ (see \cite{green1}). These elements have the
property that $\xi_{i,j} = \xi_{k,h}$ if and only if $\left( i,j \right)\sim
\left( k,h \right)$. If we eliminate repetitions in the set $\left\{\,
\xi_{i,j}
\,\middle|\, (i,j)\in I\times I \right\}$ then we obtain a basis of $S(n,r)$.
Also $S\left( B^+ \right) = \K\left\{\, \xi_{i,j} \,\middle|\, i\le j,\ \left(
i,j \right)\in I\times I \right\}$.

If $i$ has weight $\alpha\in \Lambda$, we write $\xi_{i,i}= \xi_{\alpha}$. The
set $\left\{\, \xi_{\alpha} \,\middle|\, \alpha\in \Lambda \right\}$ is a set of
orthogonal idempotents and $1_{S(n,r)} = \sum_{\alpha\in\Lambda} \xi_{\alpha}$.

A formula for the product of two basis elements is the following (see \cite{green1}):
$\xi_{i,j}\xi_{k,h} = 0$, unless $j\sim k$; and
\begin{equation}
	\label{0.0}
	\xi_{i,j}\xi_{j,h} = \sum_{\sigma} \left[ \Sigma_{i\sigma,
	h}:\Sigma_{i\sigma, j, h} \right] \xi_{i\sigma, h}
\end{equation}
where the sum is over a transversal $\left\{ \sigma \right\}$ of the set of all double cosets
$\Sigma_{i,j}\sigma \Sigma_{j,h}$ in $\Sigma_j$.

\begin{observation}
	\label{observation2}
	\begin{enumerate}
		\item $\xi_{\alpha}\xi_{i,j} = \xi_{i,j}$ or zero, according to
			$i\in \alpha$ or $i\not\in \alpha$. Similarly,
			$\xi_{i,j}\xi_{\beta} = \xi_{i,j}$ or zero, according
			to $j\in \beta$ or $j\not\in \beta$.
		\item		 If $\Sigma_{i,j} \Sigma_{j,h} = \Sigma_j$, then the
 	product $\xi_{i,j}\xi_{j,h}$ is a scalar multiple of $\xi_{i,h}$.
	\end{enumerate}
\end{observation}
Here we are particularly interested in products of the type
$\xi_{l\left( \nu,m \right),l}\xi_{l,j}$, for $l = l\left( \lambda \right)$,
and $j\in J\left( \lambda \right)$, for some $\lambda\in \Lambda$.

\begin{lemma}
	\label{lemma1}
	Let $\lambda\in \Lambda$, $\nu \in \mathbf{n-1}$, $0\le m\le
	\lambda_{\nu + 1}$, and $j\in J\left( \lambda \right)$. If the
$\nu+1$-st row of $T_j^{\lambda}$ is constant then
	$\Sigma_{l\left( \nu, m \right), l} \Sigma_{l,j} = \Sigma_l$.
\end{lemma}
\begin{proof}
	We have $\Sigma_l = \Sigma_\lambda$.  Now we know that
	$\Sigma_{l\left( \nu,m \right)}$ differs from  $\Sigma_\lambda$ only in
	factors $\nu$ and $\nu+1$
	and in these it is
	\begin{equation*}
		\Sigma_{\{t+1,\dots, t+\lambda_{\nu}+m\}}\times \Sigma_{\{t+
		\lambda_\nu + m+1, \dots, t + \lambda_\nu + \lambda_{\nu+1}\}},
	\end{equation*}
		where $t = \lambda_1 + \dots + \lambda_{\nu -1}$. It
	follows that the intersection $\Sigma_{l\left( \nu,m \right), l}$
	differs from $\Sigma_\lambda$ only in factor $\nu+1$ and this is
	\begin{equation*}
		\Sigma_{\{t + \lambda_\nu + 1, \dots, t+ \lambda_\nu +
		m\}}\times
		\Sigma_{\{t + \lambda_\nu + m + 1, \dots, t+ \lambda_\nu +
		\lambda_{\nu + 1}\}}.
	\end{equation*}
		We can write $\Sigma_{l,j} = U_1\times \dots \times U_n$, where
	$U_s$ is a subgroup of $\Sigma_{\lambda_s}$. Therefore the product
	$\Sigma_{l\left( \nu,m \right), l} \Sigma_{l,j} = \Sigma_l$ { when the product of the two $\left( \nu+1 \right)$-st factors is
	$\Sigma_{\lambda_{\nu+1}}$. This holds if  $U_{\nu+1} =
	\Sigma_{\lambda_{\nu+1}}$, i.e.,  if the $\left( \nu+1
	\right)$-st row of $T^\lambda_j$ is constant.}
\end{proof}
\begin{lemma}
	\label{lemma2}
	Let $\lambda\in\Lambda$, $\nu\in \mathbf{n-1}$, $0\le m\le
	\lambda_{\nu+1}$. Given $j\in J\left( \lambda \right)$, suppose that the
	$\left( \nu+1 \right)$-st row of $T^\lambda_j$ is constant with all
	entries equal to $c$, and that $c$ occurs exactly $a$ times in row
	$\nu$. Then
	\begin{equation*}
		\xi_{l\left( \nu,m \right), l}\xi_{l,j} = \binom{a+m}{m}
		\xi_{l\left( \nu,m \right), j}.
	\end{equation*}
If $\nu=n-1$ then the hypothesis holds for all $j\in J(\lambda)$.
\end{lemma}
\begin{proof}
	From Lemma~\ref{lemma1} and  Observation~\ref{observation2}, we know
	that
	\begin{equation*}	
		\xi_{l\left( \nu,m \right),l}\xi_{l,j} = \left[ \Sigma_{l\left(
	\nu,m
	\right), j} : \Sigma_{l\left( \nu,m \right),l,j} \right] \xi_{l\left(
	\nu,m \right), j}.
	\end{equation*}
	Now $\Sigma_{l\left( \nu,m \right),j}$ and
	$\Sigma_{l\left( \nu,m \right), j,l}$ differ only in factors $\nu$ and
	$\nu+1$. If the entries of row $\nu$ of $T^\lambda$ where $c$ occurs in
	$T^\lambda_j$ are $t_1$, \dots, $t_a$, then factors $\nu$ and $\nu+1$ of
	$\Sigma_{l\left( \nu,m \right), j}$ and $\Sigma_{l\left( \nu,m
	\right), l,j}$ are, respectively,
	\begin{equation*}	
		\cdots \times \Sigma_{\{t_1, \dots,
		t_a, \lambda_1 + \dots + \lambda_{\nu}+1, \dots, \lambda_1 +
		\dots +
		\lambda_{\nu }+m\}} \times \cdots
	\end{equation*}
		and
		\begin{equation*}
	\cdots \times
		\Sigma_{\{t_1,\dots,
		t_a\}}\times \Sigma_{\{\lambda_1, \dots, \lambda_{\nu}+1, \dots,
		\lambda_{1} + \dots + \lambda_{\nu} + m\}}\times \dots.
		\end{equation*}
		Therefore
$\left[ \Sigma_{l\left( \nu,m \right),j} : \Sigma_{l\left( \nu,m \right),l,j} \right] = \binom{a + m}{m}$.
\end{proof}
Given $\lambda\in \Lambda$, let $\K_\lambda$ denote the one-dimensional
$S\left( B^+ \right)$-module $\K$, where $\xi_\lambda$ acts as identity and all
the other basis elements, $\xi_{i,j}$, where $i\le j$ and $\left( i,j
\right)\not\sim\left(
l,l \right)$, act as zero. It is well known (see \cite{aps}), that:
\begin{enumerate}
	\item $\left\{\, \K_\lambda \,\middle|\, \lambda\in \Lambda \right\}$ is
		a full set of irreducible $S\left( B^+ \right)$-modules.
	\item The module $S\left( B^+ \right)\xi_\lambda$ is a projective cover of
		$\K_\lambda$.
	\item The module $S\left( B^+ \right)\xi_\lambda$ has a $\K$-basis $\left\{\,
		\xi_{i,l}
		\,\middle|\, i\in I\left( \lambda \right) \right\}$.
	\item The module $\K_\lambda$ is  projective  if and
		only if $\lambda = \left( r,0,\dots, 0 \right)$.
	This is a consequence of $\#I\left(
		\lambda
		\right) = 1$ if and only if $\lambda = \left( r,0,\dots, 0 \right)$.
\end{enumerate}
To calculate an Auslander-Reiten sequence ending with $\K_\lambda$, we need to know
the first two steps of a minimal projective resolution of $\K_\lambda$. For
this, define
\begin{align*}
	P_0 &:= S\left( B^+ \right)\xi_\lambda; &
	P_1 &:=
	\begin{cases}
		\bigoplus\limits_{\nu\in \mathbf{n-1}} S\left( B^+
		\right)\xi_{\lambda\left( \nu, 1 \right)}, & \mbox{if $\ch \K
		=0$;}\\[1ex]
		\bigoplus\limits_{\nu\in \mathbf{n-1}} \bigoplus\limits_{1\le p^{d_\nu}\le
		\lambda_{\nu+1}} S\left( B^+ \right) \xi_{\lambda\left( \nu, p^{d_\nu} \right)}
		, & \mbox{ if $\ch\K =p$.}
	\end{cases}
\end{align*}
Then (see \cite[(5.4)]{aps}) the first two steps of a minimal projective
resolution of $\K_{\lambda}$ are
\begin{equation}
	\label{eq03}
	P_1\stackrel{p_1}{\longrightarrow} P_0 \stackrel{p_0}{\longrightarrow}
	\K_\lambda \to 0.
\end{equation}
Here $p_0$ is the $S(B^+)$-homomorphism defined on the generator by
{ $p_0(\xi_{\lambda}) = 1$}. The $S(B^+)$-homomorphism $p_1$ is  defined
on generators by
$p_1(\xi_{\lambda(\nu,1)})= \xi_{l(\nu, 1),l},
$
when char$(\K)=0$, and
$
p_1(\xi_{\lambda(\nu,p^{d_{\nu}})})= \xi_{l(\nu, p^{d_{\nu}}),l},
$
when char$(\K)=p$.

Notice that this determines the quiver of the algebra $S(B^+)$.
In fact there is an arrow from the vertex of $\lambda$ to  the vertex of $\mu$ if and only if $S(B^+)\xi_{\mu}$ occurs
as a summand of $P_1$.

\section{Auslander-Reiten sequences}
\label{almostsplit}
In this section we give an overview of some results and definitions connected to
the notion of Auslander-Reiten sequences.
Let $A$ be a finite dimensional algebra over~$\K$.

A short exact sequence
\begin{equation*}
	(E)\ \  \ 0 \to N \stackrel{f}{\longrightarrow } E
	\stackrel{g}{\longrightarrow } S \to 0
\end{equation*}
is said to be \emph{Auslander-Reiten}  if
\begin{enumerate}[(i)]
	\item  (E) is not split;
	\item the modules $S$ and $N$ are indecomposable;
	\item if $X$ is an indecomposable $A$-module and $h\colon X\to S$ is a
		non-invertible
		homomorphism of $A$-modules, then $h$ factors through $g$.
\end{enumerate}
\begin{theorem}[\cite{auslanderreiten}]
	Given any non-projective indecomposable $A$-module $S$, there is an
	Auslander-Reiten sequence $(E)$ ending with $S$. Moreover, $(E)$ is
	determined by $S$, uniquely up to isomorphism of short exact sequences.
\end{theorem}

 In this paper we will  construct an Auslander-Reiten sequence ending with
$\K_\lambda$, for a large number of $\lambda\in \Lambda(n,r)$.
 We will use two contravariant functors
\begin{equation*}
	D,\ (\cdot)^t\colon \mathrm{mod}S(B^+) \to \mathrm{mod}S(B^+)^{op}
\end{equation*}
where  for every $X\in \mathrm{mod}S(B^+)$
\begin{align*}
	X^t &:= \Hom_{S(B^+)}(X, S(B^+)),  & D X  & :=
	\Hom_\K(X,\K).
\end{align*}
Recall that $S(B^+)$ acts on the right of $X^t$ and $DX $,
respectively, by
$
	\left( \phi \xi \right) \left( x \right) = \phi\left( x \right) \xi$ and
$	\left( \psi \xi \right)\left( x \right)  =\psi\left( \xi x \right),$
where $\phi\in X^t$, $\psi\in DX $, $\xi \in S\left( B^+
\right)$, and $x\in X$.

Consider the Nakayama functor~\cite[p.10]{gabriel}
\begin{equation*}
	D(\cdot)^t\colon \mathrm{mod}S(B^+) \to \mathrm{mod}S(B^+).
\end{equation*}
This is a covariant right exact functor which turns projectives into injectives.

Fix $\lambda\in \Lambda(n,r)$, $\lambda \not= \left( r,0,\dots, 0 \right)$. Then
$\K_\lambda$ is indecomposable and non-projective. Consider the first two steps
of the minimal projective resolution~\eqref{eq03} of $\K_\lambda$
\begin{equation*}
	P_1\stackrel{p_1}{\longrightarrow}  P_0\stackrel{p_0}{\longrightarrow}
	\K_\lambda \to 0 .
\end{equation*}
Applying  the Nakayama functor we get from this the exact sequence
\begin{equation}
	\label{firstses}
	0 \to \tau \K_\lambda \to DP_1^t \stackrel{Dp_1^t}{\longrightarrow} DP_0^t
	\stackrel{Dp_0^t}{\longrightarrow} D\K_\lambda^t\to 0,
\end{equation}
where $\tau\K_\lambda= \Ker Dp_1^t\cong D\left( \Coker p_1^t \right)$, that
is $\tau$ is the Auslander-Reiten translation.

Given an $S\left( B^+ \right)$-homomorphism $\theta\colon \K_\lambda \to DP_0^t$, consider the short
exact sequence obtained from \eqref{firstses} by pullback along $\theta$:
\begin{equation}
	\label{secondses}
	0 \to \tau\K_\lambda \stackrel{f}{\longrightarrow} E\left( \theta
	\right) \stackrel{g}{\longrightarrow} \K_\lambda\to 0.
\end{equation}
Here $E\left( \theta \right) = \left\{\, (z,c)\in DP_1^t\oplus \K_\lambda
\,\middle|\, Dp_1^t\left( z \right) = \theta\left( c \right) \right\}$ is an
$S(B^+)$-submodule of $DP_1^t \oplus \K_\lambda$, and $f$, $g$ are the homomorphisms
of $S(B^+)$-modules defined by
$g\left( z,c \right) = c$, $f\left( v \right) = \left( v,0 \right)$,
for all $z\in Dp_1^t$, $v\in \A\K_\lambda$, and $c\in \K_\lambda$.
If we choose an appropriate $\theta$,
then~\eqref{secondses} is an Auslander-Reiten sequence. In our case, $DP_0^t$ has simple socle isomorphic to the one-dimensional
module $\K_{\lambda}$, so one can take for
$\theta$ any non-zero $S(B^+)$-ho\-mo\-mor\-phism.  Before constructing  Auslander-Reiten sequences,
we will determine $\K$-bases of $P_0$ and $P_1$ adapted to our
calculations.

Notice first that $ \left( S\left( B^+ \right)\xi_\alpha \right)^t$ and
$\xi_\alpha S\left( B^+ \right)$ are isomorphic right $S\left( B^+
\right)$-modules for every $\alpha\in \Lambda$. So we will identify these two $S\left( B^+ \right)$-modules.
We will also identify $\left( \bigoplus_{\alpha\in \Lambda'} S\left( B^+
\right)\xi_\alpha \right)^t$ with $\bigoplus_{\alpha\in \Lambda'} \xi_\alpha
S\left( B^+ \right)$, for every family $\Lambda'$ of elements in $\Lambda$.
\begin{lemma}
	\label{lemma21}
	Let $\alpha\in \Lambda$. Then $\left\{\, \xi_{l\left( \alpha \right), j}
	\,\middle|\, j\in J\left(\alpha  \right)
	\right\}$ is a $\K$-basis of $\xi_\alpha S\left( B^+ \right)$.
\end{lemma}
\begin{proof}
	We know that $\xi_\alpha S\left( B^+ \right)$ is spanned by $\left\{\,
	\xi_{l\left( \alpha \right),j} \,\middle|\,  j\in I\left(n,r
	\right),\ j\ge l\left( \alpha \right) \right\}$. As $\xi_{l\left(
	\alpha\right), j
	} = \xi_{l\left( \alpha \right), i}$ if and only if $i\pi = j$,
	for some $\pi$ in the stabilizer of $l\left( \alpha \right)$ in
	$\Sigma_r$ and this stabilizer coincides with the row stabilizer of
	$T^\alpha$, the result follows.
\end{proof}
Fix $\lambda\in \Lambda\left( n,r \right)$ and consider the result of
the application of $(\cdot)^t$ to  \eqref{eq03}. Then
$
	P_1^t \cong
		\bigoplus\limits_{\nu=1}^{n-1} \xi_{\lambda\left( \nu,1 \right)}
		S\left( B^+ \right), $      or
	$P_1^t \cong	\bigoplus\limits_{\nu=1}^{n-1} \bigoplus\limits_{1\le p^{d_\nu} \le
		\lambda_{\nu+1}}\xi_{\lambda\left( \nu, p^{d_\nu} \right)} S\left( B^+ \right),$  according as  $\ch \K = 0$ , or $\ch \K
		=p.$   Thus a $\K$-basis of $P_1^t$ is given by

\begin{equation}
	\label{1.0}
	\begin{aligned}
		B_1 &:= \left\{\, \xi_{l\left( \nu,1 \right), j}
	\,\middle|\, j\in J\left( \lambda\left(\nu,1  \right)
	\right),\ \nu\in \mathbf{n-1} \right\}, &\mbox{if $\ch\K = 0$,}\\[3ex]
	B_2 &:= \left\{\,
	\xi_{l(\nu, p^{d'}), j}
	\,\middle|\,
	\begin{aligned}	
		&j\in J( \lambda( \nu, p^{d'})),\\[1ex]&
		1\le p^{d'}\le \lambda_{\nu+1},\ \nu\in \mathbf{n-1}
		\end{aligned}
	\right\},&
	\mbox{if $\ch \K = p$.}
	\end{aligned}
\end{equation}
With the above identifications of the projective modules, the map $p^t_1\colon
P_0^t\to P_1^t$ becomes $p_1^t\left( \eta \right)
 =
	 \sum\limits_{\nu=1}^{n-1} \xi_{l(\nu,1),l}\eta, $  or
$p_1^t\left( \eta \right)
 = \sum\limits_{\nu=1}^{n-1}\sum\limits_{1\le p^{d'}\le \lambda_{\nu+1}}
	 \xi_{l(\nu,p^{d'}),l} \eta, $ according as $\ch\K =
	 0,$ or $\ch\K = p.$

To construct an Auslander-Reiten sequence ending with $\K_\lambda$, it is convenient
to obtain, from $B_1$ and $B_2$, new bases for $P_1^t$ containing $p_1^t\left(
\xi_{l\left( \lambda \right),j} \right)$, $j\in J\left( \lambda \right)$.
Suppose $\lambda$ satisfies conditions
\begin{equation}
	\begin{cases}
		\lambda_n \not=0, &\mbox{if $\ch\K=0$,}\\
		\lambda_n\not=0,\ \lambda_{n-1} <  p^{d+1}-1, & \mbox{if
		$\ch\K=p$ and $p^d\le \lambda_n<p^{d+1}$.}
	\end{cases}
	\label{cond}
\end{equation}

Given $j\in J\left( \lambda \right)$, as $j\ge l=l\left( \lambda \right)$, the
$n$th row of $T^\lambda_j$ is constant with all entries equal to $n$, and its
$\left( n-1 \right)$st row has $a$ entries equal to $n$ and $\lambda_{n-1}-a$
entries equal to $n-1$, for some $0\le a\le \lambda_{n-1}$.

We shall look first at the case $\ch\K=0$. Then by Lemma~\ref{lemma2}
\begin{equation*}
	\xi_{l\left( n-1,1 \right),l}\xi_{l,j} = \left( a+1
	\right)\xi_{l\left( n-1,1 \right),j},
\end{equation*}
and $a+1\not=0$. Using this, we shall show that we can
replace $\xi_{l\left( n-1,1 \right),j}$ by $p_1^t\left( \xi_{l,j} \right)$ in
$B_1$ and obtain a new basis for $P_1^t$.
Notice that $J\left( \lambda \right) \subset J\left( \lambda\left( n-1,1
\right) \right)$, and so $\xi_{l\left( n-1,1 \right),j}\in B_1$. On the other
hand, $\xi_{l,j}\in P_0^t=\xi_\lambda S\left( B^+ \right)$, and
\begin{equation*}
	p_1^t\left( \xi_{l,j} \right) =
\xi_{l\left( n-1,1 \right),l}\xi_{l,j}+
	\sum_{\nu=1}^{n-2} \xi_{l\left( \nu,1 \right), l}\xi_{l,j}
	 = \left( a+1
	\right)\xi_{l\left( n-1,1 \right),j}+ \sum_{\nu=1}^{n-2}
	\xi_{l\left( \nu,1 \right),l} \xi_{l,j}  .
\end{equation*}
Now $\xi_{l\left( \nu,1 \right),l}\xi_{l,j}$ is a linear combination of basis
elements of the type $\xi_{l\left( \nu,1 \right),j\pi}$, for some $\pi\in
\Sigma_r$. As, for $\nu=1,\dots, n-2$, we have that $l\left( \nu,1
\right)\not\sim l\left( n-1,1 \right)$, we get that $\xi_{l\left( n-1,1
\right),j}$ is always different from $\xi_{l\left( \nu,1 \right),i}$, for any
$i\in I\left( n,r \right)$. Therefore, we can replace $\xi_{l\left( n-1,1
\right),j}$ in $B_1$ by $p_1^t\left( \xi_{l,j} \right)$ and still get a basis
for $P_1^t$. We have proved
the following:
\begin{proposition}
	\label{lemma2(2)}
	If $\ch\K=0$ and $\lambda_n\not=0$, then
	\begin{align*}
		\overline{ B }_1 =& \left\{\, \xi_{l\left( \nu,1 \right),j}
		\,\middle|\, j\in J\left( \lambda\left( \nu,1 \right) \right),
		\ \nu=1,\dots, n-2 \right\}\\ &\cup
		\left\{\, \xi_{l\left( n-1,1 \right),j} \,\middle|\, j\in
		J\left( \lambda\left( n-1,1 \right) \right)\setminus J\left(
		\lambda \right) \right\} \cup
		\left\{\, p_1^t\left( \xi_{l,j} \right) \,\middle|\, j\in
		J\left(\lambda  \right) \right\}
	\end{align*}
	is a basis of $P_1^t$. In particular, $p_1^t$
	is a monomorphism.
\end{proposition}

Suppose now that $\ch\K=p$ and that $\lambda$ satisfies condition~\eqref{cond}.
We will apply Lemma~\ref{lemma2}, together with the following well known consequence of Lucas' Theorem:
\begin{proposition}
	\label{lucas}
	Assume $m$ and $q$ are positive integers and that
	\begin{align*}
		m = m_0 + m_1 p + \dots + m_tp^t \,\,\,\ \mbox{and}\,\,\,\,
		q = q_0 + q_1 p + \dots + q_sp^s
	\end{align*}
	are the $p$-adic expansions of $m$ and $q$.
	Then $p$ divides $\binom{m}{q}$ if and only if $m_\nu<q_\nu$ for some
	$\nu$.
\end{proposition}

Given $j\in J\left( \lambda \right)$, we denote by $a=a\left( j \right)$ the
number of entries equal to $n$ in the $(n-1)$st row of $T^\lambda_j$.
Let $a_0$, \dots, $a_d$ be the coefficients in the $p$-adic expansion of
$a$, with $a_d$ possibly equal to zero. As $a\le \lambda_{n-1}<p^{d+1}-1$, and all the
coefficients in the $p$-adic expansion of $p^{d+1}-1$ are equal to $p-1$, we get
that there is some $a_t\not=p-1$. Let
\begin{equation}
\label{cond3.5}
	m\left( j \right) := \min\left\{\, t \,\middle|\, a_t<p-1 \right\}
\end{equation}
and define
\begin{equation*}
	J\left(\lambda, d' \right) = \left\{\, j\in J\left( \lambda \right)
	\,\middle|\, m\left( j \right) = d'
	\right\}.	
\end{equation*}
Obviously $J\left(\lambda  \right) = \dot\bigcup_{0\le d'\le d} J\left( \lambda,
d'
\right)$.
\begin{proposition}
	\label{lemma3}
	Suppose that $\ch\K=p$
	and that $\lambda\in \Lambda\left( n,r \right)$ satisfies
	condition~\eqref{cond}.
	Then
\begin{align*}
	\overline{ B }_2 = &
	\left\{\, \xi_{l\left( \nu,p^{d'} \right), j} \,\middle|\, j\in J(
	\lambda( \nu, p^{d'} ) ),\ 1\le p^{d'}\le
	\lambda_{\nu+1},\ \nu=1,\dots, n-2\right\}
	\\&\cup
	\left\{\, \xi_{l\left( n-1,p^{d'} \right),j} \,\middle|\, j\in J(
	\lambda( n-1,p^{d'} ) )\setminus J\left( \lambda,d'
	\right),\ 0\le d'\le d \right\} \\&\cup
	\left\{\, p_1^t\left( \xi_{l,j} \right) \,\middle|\, j\in J\left(
	\lambda
	\right) \right\}
\end{align*}
is a $\K$-basis for $P_1^t$. In particular, $p_1^t$ is a monomorphism.
\end{proposition}
\begin{proof}
Let $j\in J\left( \lambda \right)$. Just like in the characteristic zero case,
we consider the basis element $\xi_{l,j}$ of $P_0^t$ and look at
\begin{equation*}
	p_1^t\left( \xi_{l,j} \right) = \sum_{\nu=1}^{n-1}\sum_{1\le
	p^{d'}\le \lambda_{\nu+1}} \xi_{l\left( \nu, p^{d'} \right),l}
	\xi_{l,j}.
\end{equation*}
 By Lemma~\ref{lemma2} for any $d'$ such that $0\le p^{d'}\le \lambda_{n}$
we have
\begin{equation*}
	\xi_{l\left( n-1,p^{d'} \right),l} \xi_{l,j} =
	\binom{a + p^{d'}}{p^{d'}} \xi_{l\left( n-1,p^{d'}
	\right),j}.
\end{equation*}
Here $a$ is the number of times $n$ occurs in row $n-1$ of $T_j^{\lambda}$.
It follows from Proposition~\ref{lucas} and the definition of $m\left( j
\right)$ (see~\eqref{cond3.5} ), that
$p$ does not divide $\binom{a+p^{m\left( j \right)}}{p^{m\left( j \right)}}$
and divides all $\binom{a+p^{d'}}{p^{d'}}$ for $d'<m\left( j \right)$.
Therefore
\begin{equation}
	\label{1.4}
	p^t_1\left( \xi_{l,j} \right)
	= \sum_{\nu=1}^{n-2}\sum_{1\le p^{d'}\le \lambda_{\nu+1}} \xi_{l\left(
	\nu,p^{d'} \right),l}\xi_{l,j} +
	\sum_{p^{m(j)}\le p^{d'}\le \lambda_n}
	\binom{a+p^{d'}}{p^{d'}} \xi_{l\left( n-1,p^{d'} \right),j}
\end{equation}
and the coefficient of $\xi_{l\left( n-1,p^{m\left( j \right)} \right),j}$ in
this sum is non-zero. As, for $\nu\not=n-1$, we have $l\left( n-1,p^{m\left(
j \right)} \right)\not\sim l\left( \nu,p^{d'} \right)$ it follows that
$\xi_{l( n-1,p^{m\left( j \right)} ),j}$ does not appear in the
basis
expansion of $\xi_{l\left( \nu,p^{d'} \right),l}\xi_{l,h}$, for any $h\in
J\left( \lambda \right)$. Also, if
$d'\not=m\left( j \right)$, then $l\left( n-1,p^{d'} \right)\not\sim l\left(
n-1,p^{m\left( j \right)} \right)$ and so $\xi_{l\left( n-1,p^{d'}
\right),h}\not= \xi_{l\left( n-1,p^{m\left( j \right)} \right),j}$ for all
$h\in J\left(\lambda  \right)$.
Finally, suppose $h\in J\left( \lambda \right)$ satisfies
$
	(\, l( n-1,p^{m( j )} ),j \,)\sim (\,
	l(
	n-1,p^{m( h )}),h
\,).$
	Then $h=j\pi$ for some $\pi\in \Sigma_{\lambda\left( n-1,p^{m\left( j
\right)} \right)}$. But, since both $h$, $j\ge l$, we can not move any
entry $n-1$ in   row $(n-1)$ of $T^\lambda_j$ to row $n$ to obtain
$T^\lambda_h$. This implies that $\pi$ belongs to the row stabilizer of
$T^\lambda$. As both $T^\lambda_j$ and $T^\lambda_h$ are row semistandard
we get $h=j$.
Therefore $\xi_{l\left( n-1,p^{m\left( j \right)} \right),j}$ appears only once
in $\overline{ B }_2$: in the expression~\eqref{1.4} of $p_1^t\left(
\xi_{l,j}
\right)$ with the coefficient $\binom{a+p^{m\left( j \right)}}{p^{m\left( j
\right)}}$. Hence, we can replace $\xi_{l\left( n-1,p^{m\left( j \right)}
\right),j}$ by $p^t_1\left( \xi_{l,j} \right)$ in $B_2$ for all $j\in J\left(
\lambda \right)$ and still have a basis for $P_1^t$.
\end{proof}
It is now easy to obtain an Auslander-Reiten sequence ending with $\K_\lambda$ for
$\lambda$ satisfying~\eqref{cond}. Denote, respectively, by $B_1^*$ and $B_2^*$ the $\K$-basis
of $DP_1^t $ dual to $\overline{ B }_1$ and $\overline{ B }_2$.
For $j\in J\left( \lambda \right)$, we denote by $z_{l,j}$ the element in
$B_1^*$ (respectively, in $B_2^*$) that is dual to $p_1^t\left( \xi_{l,j}
\right)$.
Let $U_\lambda$ be the subspace  of $DP_1^t $
with $\K$-basis $B_1^*\setminus \left\{\, z_{l,j} \,\middle|\, j\in J\left(
\lambda
\right)\right\}$  if $\ch\K=0$ or $B_2^*\setminus \left\{\, z_{l,j}
\,\middle|\, j\in J\left( \lambda \right) \right\}$ if $\ch\K=p$. Then $U_{\lambda}$ is in fact a $S(B^+)$-submodule. Define
\begin{equation*}
	E\left( \lambda \right) = \left\{\, \left( z,c \right)\in D P_1^t
	\oplus \K_\lambda \,\middle|\,  z\in \left( U_\lambda + c
	z_{l,l} \right) \right\}.
\end{equation*}
Then we have the following result.
\begin{theorem}
	\label{prop1}
	Suppose that $\lambda\in \Lambda\left( n,r \right)$
	satisfies~\eqref{cond}.Then
	the sequence
	\begin{equation}
		\label{**}
		0\to U_\lambda \stackrel{f}{\longrightarrow} E\left( \lambda
		\right)
		\stackrel{g}{\longrightarrow} \K_\lambda \to 0,
	\end{equation}
where $f$ and $g$ are defined by $f\left( z \right) = \left( z,0 \right)$ and
	$g\left( z',c \right) = c,$ for all $ z\in U_\lambda, (z',c)\in E\left( \lambda \right),$ is an Auslander-Reiten sequence.
\end{theorem}
\begin{proof}
	Notice first that $\K_\lambda$
	is not projective, since $\lambda \not=\left( r,0,\dots,0 \right)$. Hence an Auslander-Reiten sequence ending with
	$\K_\lambda$ exists.
	
	Now we will prove the theorem in the case $\ch\K=0$. The case of
	$\ch\K=p$ is similar.

As we mentioned before, $DP_0^t $  has simple socle, isomorphic   to $\K_\lambda.$ Therefore, for any non-zero $\theta \in \Hom_{S(B^+)}(\K_\lambda,DP_0^t ),$ the sequence~\eqref{secondses} is an Auslander-Reiten sequence. We will consider $\theta$ defined by
\begin{equation*}
	\mbox{
	$\theta(c)(\eta) = \eta c,$
	for all $\eta \in P_0^t = \xi_{\lambda}S(B^+)$ and all $c\in
	\K_\lambda$.}
\end{equation*}
Note that as $P_0^t$ has $\K$-basis $\left\{\, \xi_{l,j} \,\middle|\,
	j\in J\left(\lambda  \right)
	\right\}$ and, for $j\in J\left( \lambda \right)$ and $c\in \K_\lambda$, $\xi_{l,j}c = c$ or $0$, according as $j=l$ or $j\not=l$,
	we have that $\theta$ is completely determined by saying that $\theta(c)(\xi_{l,j} ) = c$  if $j=l$, and 0 otherwise.
Given $z\in DP_1^t $ we can write $z$ as a linear combination of
the elements of $B_1^*$. Then for any $c\in \K_\lambda$, we have
$Dp_1^t(z) = \theta (c)$ if and only if $zp_1^t = \theta(c)$, which in turn
holds if and only if  for all $j\in J\left( \lambda \right)$  there holds $zp_1^t\left( \xi_{l,j} \right) = c$ if $j=l$, and 0 otherwise.
	Thus $z =c z_{l,l}+u$ for some $u\in U_\lambda$.
	Hence
	\begin{equation*}
		E\left( \theta \right) = \left\{\, (z,c)\in D P_1^t
		\oplus \K_\lambda \,\middle|\, Dp_1^t \left( z
		\right) = \theta(c)
		\right\} = E\left( \lambda \right).
	\end{equation*}
In a similar way, we see that $z\in\A\K_\lambda = \ker Dp_1^t  $
if and only if $zp_1^t = 0$, that is if and only if $z\in U_\lambda$. Therefore
$\A\K_\lambda = U_\lambda$.
\end{proof}
\begin{remark}\label{be} We have explained  that any non-zero homomorphism from the simple module $\K_{\lambda}$ into $DP_0^t$  gives an Auslander-Reiten
sequence. In particular if we replace $\theta$ by $c\theta$, where $c$ is a non-zero scalar, then this gives the same Auslander-Reiten sequence.
In fact we can say more. Recall the Auslander-Reiten formula. For any modules $X, Y$ of some algebra, we have
(see, for example, Theorem 2.20 in \cite{greenunpublished})
$${\rm Ext}^1(X, \tau Y) \cong D\underline{\rm Hom}(Y, X).
$$
Here $\underline{\rm Hom}(U, V)$ is the quotient space of ${\rm Hom}(U, V)$ modulo homomorphisms which factor through a projective module.
We apply this with $X=Y =\K_{\lambda}$.
Then the right hand side is trivially one-dimensional. Hence ${\rm Ext}^1(\K_{\lambda}, \tau\K_{\lambda})\cong \K$.
Therefore,  by the previous observation, if we have a non-split exact sequence with end terms $\K_\lambda$ and $\tau\K_{\lambda}$ this
must be an Auslander-Reiten sequence.
\end{remark}

We have constructed Auslander-Reiten sequences ending with $\K_\lambda $, for all $\lambda\in \Lambda\left( n,r \right)$ satisfying conditions~\eqref{cond}. For this we only need to deal with the multiplication of basis elements of $S\left( B^+ \right)$ where the formula in Lemma~\ref{lemma2} can be used. For $\lambda $  not satisfying~\eqref{cond}, the calculations for multiplication of basis elements get very tortuous, with many particular cases to consider,  and the method we use does not work well  in the construction of the desired sequences. Two more cases are studied: the case $n=2$ is completely treated in Section \ref{n2}, and in Section \ref{n3} we give an example for $n=3.$
 \section{The middle term of the Auslander-Reiten sequence}\label{newsection}

Given the Auslander-Reiten sequence  \eqref{**}, one would like to know
when the module $E(\lambda)$ is indecomposable. This seems to be a difficult question in general, as one can see for the Borel-Schur algebras
of finite type  (see Section \ref{finitetype}). In fact  one of the motivations for our classification
was to get a better understanding of this question.
We have two easy observations, which deal with most of the cases when the algebra has finite type. The first one involves the indecomposibility of the module $P_1$. Notice that  $P_1$ is indecomposable
if and only if $\lambda=(\lambda_1, 0, \cdots,0, \lambda_{\nu },0, \cdots,0)$, for some $2 \leq\nu \leq n $,  $\lambda_{\nu }\geq 1$, if $\ch\K=0$, and $1 \leq\lambda_{\nu }< p$, if $\ch\K=p$.
\begin{proposition} Given $\lambda \in \Lambda\left( n,r \right)$,
assume  the module $P_1$ is indecomposable.  Then the
middle term $E(\lambda)$ is indecomposable.
\end{proposition}
\begin{proof}
We construct $E(\lambda)$ as a pullback, and hence we have a commutative diagram with exact rows
\begin{equation}\label{diagramK}
\CD  0@>>> \tau \K_{\lambda} @>>> E(\lambda) @>>> \K_{\lambda} @>>> 0 \cr
&&   @V{1}VV  @V{\tilde{\theta}}VV @V{\theta}VV \cr
0@>>> \tau \K_{\lambda} @>>> DP_1^t @>{Dp_1^t}>> DP_0^t
\endCD
\end{equation}
By the Snake Lemma, the map $\tilde{\theta}$ is injective. Since $P_1$ is indecomposable, the module $DP_1^t$ is indecomposable
injective and hence has a simple socle. Therefore the socle of $E(\lambda)$ is simple, and the module is
indecomposable.
\end{proof}
  Note that  if $n=2$ and $\ch\K=0$, then $P_1$ is always indecomposable. Therefore we have the following result.
\begin{corollary} If  $n=2$ and $\ch\K=0$, then  the
middle term $E(\lambda)$ is always indecomposable.
\end{corollary}

We can also identify from \eqref{diagramK} the Auslander-Reiten sequence for $\K_{\lambda}$ when $\lambda = (0, 0, \ldots, r)$. This is the unique simple module which is
injective. It follows that $DP_0^t\cong \K_{\lambda}$ and the map $\theta$ is an isomorphism. In this case the
map $Dp_1^t$ must be onto, and then the  Auslander-Reiten sequence is equivalent to the exact sequence which is the bottom row of the diagram  \eqref{diagramK}.
\section{The case $n=2$}
\label{n2}
In this section we study the construction of an Auslander-Reiten sequence ending with
$\K_\lambda$ in the particular case of $n=2$. We will show that it is very easy to
obtain such sequences with no restriction on $\lambda$ or the characteristic of
$\K$.

Let $\lambda=\left( \lambda_1,\lambda_2 \right)$. Since $\K_\lambda$ is
non-projective if and only if $\lambda_2\not=0$, all compositions we are
interested in satisfy this condition. So in this section we assume that
$\lambda_2\not=0$. In particular, the construction of Auslander-Reiten sequences in the
characteristic zero and $n=2$ case is completely answered in
Theorem~\ref{prop1}.

Suppose now that $\ch \K = p$ and $d$ is such that $p^d \le
\lambda_2<p^{d+1}$. Given $j\in J\left( \lambda \right)$, recall that $a\left(
j \right)$ is the number of $2$'s in the first row of $T^\lambda_j$. If
\begin{equation}
	\label{2.1}
	a = a(j) = (p-1) + (p-1)p + \dots + (p-1)p^d + \dots
\end{equation}
is the $p$-adic expansion of $a$ then, by Proposition~\ref{lucas}, for all
$0\le d'\le d$ the binomial coefficient $\binom{a + p^{d'}}{p^{d'}}$ is
divisible by $p$.  Hence
\begin{equation*}
	p_1^t \left( \xi_{l,j} \right) =
	\sum_{ d' =0}^d \xi_{l\left( 1,p^{d'} \right),l} \xi_{l,j} =
	\sum_{d'=0}^d \binom{a+p^{d'}}{p^{d'}} \xi_{l\left( 1,p^{d'} \right),j}
	= 0.
\end{equation*}
Next we suppose that $a=a\left( j \right)$ has $p$-adic expansion $	a = a_0 + a_1 p + \dots + a_s p^s,$
with $a_t\not=p-1$ for some $0\le t\le d$. Define $m(j) = \min\left\{\, t
\,\middle|\, t\le d \mbox{ and } a_t<p-1 \right\}$ and
\begin{equation*}
	\hat{J}(\lambda) = \left\{\, j\in J\left( \lambda \right) \,\middle|\,
	a(j) \not= (p-1) + (p-1)p + \dots + (p-1)p^d + \dots \right\}.
\end{equation*}
For $0\le d'\le d$ we denote by $\hat{J}(\lambda,d')$ the subset of those
$j\in \hat{J}(\lambda)$ such that $m(j) = d'$. Then $\hat{J}(\lambda) =
\dot\bigcup_{0\le d'\le d} \hat{J}(\lambda,d')$ and $\hat{J}\left( \lambda,d'
\right)\subset J(\lambda(1,p^{d'}))$. Now with a proof completely analogous to
the proof of Proposition~\ref{lemma3}, we see that, for $j\in \hat{J}(\lambda,d')$, the
element $\xi_{l(1,p^{m(j)}),j}$ in $B_2$ can be replaced by
$p_1^t(\xi_{l,j})$ and the resulting set $\overline{ B }_2$ is a new basis
for $P_1^t$. This proves the following result.
\begin{proposition}
	\label{lemma4}
Suppose that $\ch\K= p$ and $\lambda=\left( \lambda_1,\lambda_2 \right)$, with
$\lambda_2\not=0$. Then
\begin{align*}
	\overline{ B }_2=&
	\left\{\, \xi_{l\left( 1,p^{d'} \right),j} \,\middle|\, j\in
	J(\lambda(1,p^{d'}))\setminus \hat{J}(\lambda,d'),\ 0\le d'\le d \right\}
	\\&\cup
	\left\{\, p_1^t\left( \xi_{l,j} \right) \,\middle|\, j\in \hat{J}\left(
	\lambda \right)\right\}
\end{align*}
is a $\K$-basis for $P_1^t$.

We also have that
 $\left\{\, \xi_{l,j} \,\middle|\, j\in J\left( \lambda
\right)\setminus \hat{J}(\lambda) \right\}$ is a $\K$-basis for $\ker (p^t_1)$.
In particular, $p^t_1$ is injective if and only if $p^d \le
\lambda_2<p^{d+1}$ and $\lambda_1< p^{d+1}-1$, i.e., if and only if $\lambda$ satisfies condition \ref{cond}.
\end{proposition}
Denote by $B_2^*$ the basis of $D P_1^t $ dual to $\overline{ B
}_2$. We write $z_{l,j}$ for the element dual to $p_1^t(\xi_{l,j})$, where
$j\in \hat{J}(\lambda)$. Let $U_\lambda$ be the $S(B^+)$-submodule of
$DP_1^t$ with $\K$-basis $B_2^*\setminus \left\{\, z_{l,j} \,\middle|\, j\in
\hat{J}(\lambda)
\right\}$ and
\begin{equation*}
	E\left( \lambda \right) = \left\{\, \left( z,c \right)\in D P_1^t
	\oplus \K_\lambda \,\middle|\, z\in (U_\lambda + cz_{l,l})
	\right\}.
\end{equation*}
Then, adapting the proof of Theorem~\ref{prop1}, we can conclude the
following result.
\begin{theorem}
	\label{thmn2}
Suppose that $\ch\K=p$ and $\lambda=\left( \lambda_1,\lambda_2 \right)$, with
$\lambda_2\not=0$. Then the sequence
\begin{equation*}
	0\to U_\lambda \stackrel{f}{\longrightarrow} E\left( \lambda \right)
	\stackrel{g}{\longrightarrow} \K_\lambda\to 0,
\end{equation*}
where $f$ and $g$ are defined by $f\left( z \right) = \left( z,0 \right)$ and
	$g\left( z',c \right) = c,$ for all $ z\in U_\lambda, (z',c)\in E\left( \lambda \right),$ is an Auslander-Reiten sequence.
\end{theorem}
\section{Some results in the case $n=3$}
\label{n3}
We will consider fields of characteristic $p$, $n=3$, and $\lambda\in
\Lambda(3,r)$ with $\lambda_3\not=0$.  Define $d$ by $p^d\le
\lambda_3<p^{d+1}$. If $\lambda_2<p^{d+1}-1$, we know from
Theorem~\ref{prop1} an Auslander-Reiten sequence ending with $\K_\lambda$. In
this section we study the construction of Auslander-Reiten sequences for $\lambda$
with $\lambda_2 = 2p^{d+1} -1$.

In this case $p_1^t$ may not be injective. Our first step will be again to
determine a basis for $P_1^t$, which contains a basis of $\im(p_1^t) $.
Recall that
\begin{align*}
	B_2 = & \left\{\, \xi_{l\left( \nu,p^{d'} \right),j} \,\middle|\, j\in
	J\left( \lambda(\nu, p^{d'}) \right),\ 1\le p^{d'}\le \lambda_{\nu+1},\
	\nu=1,2\right\}	
	\label{4.0}
\end{align*}
is a $\K$-basis of $P_1^t$. We will end this section by explaining how to
replace some of these elements $\xi_{l(\nu,p^{d'}),j}$ by elements of the form
$p_1^t(\xi_{l,h})$ and obtain a new basis of~$P_1^t$. The construction of the Auslander-Reiten sequence ending with
$\K_\lambda$ is then similar to the one in the previous sections.

Given $j\in J\left( \lambda \right)$, suppose that the number of entries
equal to $3$ in the second row of $T^\lambda_j$ is $a=a\left( j \right)$. Let
$a_0$, \dots, $a_{d+1}$ be the coefficients of the $p$-adic expansion of $a$.
If $a_t\not=p-1$, for some $t\le d$, let $m=m\left( j \right) = \min\left\{\,
t \,\middle|\, a_t<p-1 \right\}$. Then
\begin{equation*}
	p_1^t\left( \xi_{l,j} \right) =
	\sum_{d'=0}^{d+1} \xi_{l\left( 1,p^{d'} \right),l}\xi_{l,j} +
	\sum_{d'=m(j)}^{d} \binom{a+p^{d'}}{p^{d'}} \xi_{l\left( 2,p^{d'}
	\right),j}
\end{equation*}
and $p$ does not divide $\binom{a+p^{m(j)}}{p^{m(j)}}$. Now, like in the proof
of Proposition~\ref{lemma2(2)}, it is simple to see that if we replace
$\xi_{l(2,p^{m(j)}),j}$ by $p_1^t(\xi_{l,j})$ in $B_2$ for all $j$'s satisfying
these conditions, we obtain a new basis $B_2'$ for $P_1^t$.

The problem arises when the $p$-adic expansion of $a$ is
\begin{equation}
	a = (p-1) + (p-1)p + \dots + (p-1)p^d + cp^{d+1},\ c=0,1.
	\label{4.1}
\end{equation}
If $a=a(j)$ satisfies \eqref{4.1}, we say that $j$ is a \emph{critical} element of
$J\left( \lambda \right)$. In this case we have
\begin{equation*}
	p_1^t\left( \xi_{l,j} \right) = \sum_{d'=0}^{d+1}\xi_{l\left(
	1,p^{d'}
	\right),l}\xi_{l,j}.
\end{equation*}
If $c=0$, then the second row of $T^\lambda_j$ is not constant and we can not
use the multiplication formula in Lemma~\ref{lemma2} to calculate
$\xi_{l(1,p^{d'}),l}\xi_{l,j}$. We will use a different version of the
multiplication formula \eqref{0.0} to study these products (see
\cite[(2.7)]{green1}).
Given $i$, $j$, $k\in I\left( 3,r \right)$
the double cosets $\Sigma_{i,j} \sigma \Sigma_{j,k}$ in $\Sigma_j$ correspond
one-to-one to the $\Sigma_{j,k}$-orbits of $i\Sigma_j$. So \eqref{0.0} becomes
\begin{equation}
	\xi_{i,j}\xi_{j,k} = \sum_{h}\left[ \Sigma_{h,k}:\Sigma_{h,j,k}
	\right]\xi_{h,k},
	\label{4.2}
\end{equation}
where the sum is over a transversal $\left\{ h \right\}$ of the
$\Sigma_{j,k}$-orbits in the set $i\Sigma_j$. Now we fix a critical $j\in J\left( \lambda
\right)$  such that $a=a\left( j \right) = p^{d+1}-1$. Suppose that the number
of entries equal to $2$ and the number of entries equal to $3$ in the first row
of $T^\lambda_j$ are $t_2 = t_2\left(
j \right)$ and  $t_3 = t_3\left( j \right)$, respectively.
Thus we have
\begin{equation}
T_j^\lambda=
\raisebox{+1.8ex}{
\ensuremath{
\begin{array}{l}
	1\dots 1 \overbrace{2\dots\dots 2}^{t_2} \overbrace{3\dots\dots \dots
	3}^{t_3}\\[3ex]
	2\dots\dots 2\underbrace{3\dots\dots\dots\dots\dots\dots 3}_a\\[2ex]
	3\dots\dots\dots\dots\dots 3
\end{array}}}
	\label{4.3}
\end{equation}
Applying \eqref{4.2} in the case of our composition, we obtain
\begin{equation}
	\xi_{l\left( 1,p^{d'} \right),l}\xi_{l,j} = \sum_{h}
	\binom{t_2 + s}{s} \binom{t_3 + t}{t} \xi_{h,j}
	\label{4.4}
\end{equation}
where
\begin{equation*}
	h = ( \underbrace{1,\dots,1}_{\lambda_1+s},
	\underbrace{2,\dots,2}_{\lambda_2-a-s}, \underbrace{1,\dots,1}_t,
	\underbrace{2,\dots,2}_{a-t}, \underbrace{3,\dots,3}_{\lambda_3})
\end{equation*}
and $s+t =p^{d'}$, $t<p^{d+1}$.  Note that all the $\xi_{h,j}$ in
\eqref{4.4} are distinct.
\begin{remark}
	\label{4.5}
	Suppose that $j_1$ and $j_2$ are critical elements of
$J\left( \lambda \right)$. Then $\left( h,j_1 \right)\sim \left( h',j_2
\right)$ implies $j_1\sim j_2$. Hence if $j_1\not\sim j_2$ all the basis
elements $\xi_{h,j_1}$ appearing in the $B_2$-expansion of $p_1^t\left( \xi_{l,j_1}
\right)$ are distinct from those appearing in the $B_2$-expansion of
$p_1^t\left( \xi_{l,j_2} \right)$.
\end{remark}
Thus given $j$ defined by \eqref{4.3}, we only have to study the linear
independence of $\left\{ p_1^t\left( \xi_{l,j}\right),\ p_1^t\left( \xi_{l,j'} \right)
 \right\}$, where $j'$ is a critical element of $J\left( \lambda \right)$
and $j'\sim j$.
Hence
\begin{equation}
	\label{4.6}
	T^\lambda_{j'} =
	\raisebox{1.7ex}{
	$\begin{array}{l}
		1\dots 1 \overbrace{2\dots\dots\dots 2}^{t_2 + p^{d+1}}
		\overbrace{3\dots\dots 3}^{t_3 -p^{d+1}}\\
		3\dots \dots \dots \dots \dots 3\\
		3\dots \dots \dots 3
	\end{array}$
	}.
\end{equation}
Note that if $t_3<p^{d+1}$, then $j'$ and $p_1^t\left( \xi_{l,j'} \right)$ are not
defined. Thus we will assume that $t_3\ge p^{d+1}$.

Recall that, from Lemma~\eqref{lemma2}, we have
\begin{equation}
	\label{4.7}
	p_1^t(\xi_{l,j'}) = \sum_{0\le d'\le d+1}
	\binom{t_3 -p^{d+1}+p^{d'}}{p^{d'}}\xi_{l\left( 1,p^{d'} \right),j'}.
\end{equation}
Before we proceed we need a technical result. Its proof is an easy consequence
of Proposition~\ref{lucas}.
\begin{lemma}
	\label{prop4.8}
	Let $0\le m\le d+1$. Then $p$ divides  all the products
	$\binom{t_2+s}{s}\binom{t_3+t}{t}$ with  $s+t = p^{d'}$ and $0\le d'\le
	m$
	if and only if the $p$-adic expansions of $t_2$ and $t_3$ have the form
	\begin{equation}
		\label{t2t3}
	\begin{aligned}
		t_2& = (p-1) + (p-1)p + \dots+  (p-1)p^m + c'_{m+1} p^{m+1} +
		\dots \\
		t_3 &= (p-1) + (p-1) p + \dots + (p-1) p^m + c''_{m+1}
		p^{m+1} + \dots.
	\end{aligned}
\end{equation}
\end{lemma}
\begin{proof}
Suppose $p$ divides all the products $\binom{t_2+s}{s}\binom{t_3+t}{t}$ with
$s+t = p^{d'}$ and $0\le d'\le m$. Taking $s=p^{d'}$ and $t=0$
we get $\binom{t_3 + t }{t}=1$. Thus, $\binom{t_2 + p^{d'}}{p^{d'}}$ is
divisible by $p$ for any $0\le d'\le m$. It follows from
Proposition~\ref{lucas} that the  coefficient of $p^{d'}$ in the $p$-adic expansion of
$t_2$ is $(p-1)$ for any $0\le d'\le m$. The case of $t_3$ is proved similarly.

Now, suppose that $t_2$ and $t_3$ satisfy \eqref{t2t3} and $0\le d'\le p^m$,
$s+t = p^{d'}$. If $s=0$, then $t=p^{d'}$ and $\binom{t_3 +p^{d'}}{p^{d'}}$ is
divisible by $p$ by Proposition~\ref{lucas}. Suppose $s\not=0$ and $s_i$ is the
first non-zero coefficient  in the $p$-adic expansion of $s$. Then the
$i$th coefficient of $t_2 + s$ in its $p$-adic expansion is $s_i-1$. Since
$s_i-1<s_i$, we get, from Proposition~\ref{lucas}, that $\binom{t_2 + s}{s}$ is
divisible by $p$.
\end{proof}
As an immediate consequence of Lemma~\eqref{prop4.8} we get the following
result.
\begin{lemma}
	\label{lemma4.1}
	Given $j$ and $j'$ as above we have:
	\begin{enumerate}[(i)]
		\item $p_1^t\left( \xi_{l,j} \right) = 0$ if and only if the
			$p$-adic expansions of $t_2$ and $t_3$ are
			\begin{align*}
				t_2& = (p-1) + (p-1)p + \dots+
				(p-1)p^{d+1} + c'_{d+2} p^{d+2} +
		\dots \\
		t_3 &= (p-1) + (p-1) p + \dots + (p-1) p^d + c''_{d+1}
		p^{d+1} + \dots;
			\end{align*}
		\item $p_1^{t}\left( \xi_{l,j'} \right) =0$ if and only if the
			$p$-adic expansion of $t_3$ is of the form
			\begin{equation*}
				(p-1) + (p-1)p + \dots + (p-1) p^{d} + 0\cdot
				p^{d+1} + \dots .
			\end{equation*}
	\end{enumerate}
\end{lemma}
Suppose now that $p_1^t\left( \xi_{l,j} \right)\not=0$. Let
\begin{equation*}
	b:= b(j):= \min\left\{\,  0\le d'\le d+1 \,\middle|\,
	\xi_{l(1,p^{d'}),l} \xi_{l,j}\not=0 \right\}.
\end{equation*}
Then $p$ divides all the products $\binom{t_2 + s}{s}\binom{t_3 + t}{t}$ with
$s+t = p^{d'}$ and $0\le d'< b$,  and there are $s$ and $t$ such that $s+t= p^b$
and $\binom{t_2 + s}{s} \binom{t_3 + t}{t}$ is non-zero in $\K$.
From Lemma~\ref{prop4.8}, we obtain that the $p$-adic expansions of $t_2$ and
$t_3$ should be of the form
\begin{equation}
	\label{4.9}
	\begin{aligned}
		t_2 & = (p-1) + \dots + (p-1)p^{b-1} + c_b' p^b + \dots \\
		t_3 & = (p-1) + \dots + (p-1)p^{b-1}+  c_b'' p^b + \dots,
	\end{aligned}
\end{equation}
where either $c_b'$ or $c_b''$ is different from $p-1$.

\begin{lemma}
	\label{lemma4.2}
Given $j$, $j'$ and $b=b(j)$ as above, we have
$\xi_{l\left( 1,p^{d'} \right), l}\xi_{l,j}=0$, if $d'<b$, and
\begin{equation*}
	\xi_{l(1,p^b),l} \xi_{l,j} =
	\begin{cases}
		\binom{t_2 + p^b}{p^b} \xi_{l\left( 1,p^b \right),j} +
		\binom{t_3 + p^b}{p^b} \xi_{i,j}, & \mbox{if $b<d+1$}\\[2ex]
		\binom{t_2 + p^{d+1}}{p^{d+1}} \xi_{l\left( 1,p^{d+1}
		\right),j}, & \mbox{ if $b=d+1$},
	\end{cases}
\end{equation*}
where
\begin{equation*}
T^\lambda_i =
\raisebox{0.0ex}{
$\begin{array}{l}
1\dots\dots\dots 1\\[1ex]
\underbrace{2\dots\dots 2}_{p^{d+1}} \underbrace{1\dots \dots 1}_{p^b} 2\dots
\dots \dots 2\\
3\dots \dots \dots 3
\end{array}$
}.
\end{equation*}
\end{lemma}
\begin{proof}
In the conditions of the lemma, we see that $t_2$ and $t_3$ are as in
\eqref{4.9}. Suppose $b<d+1$.  Given $0<s<p^b$, if $s_i$ is the first
non-zero coefficient in the  $p$-adic expansion of $s$, then the $i$th
coefficient in the $p$-adic expansion of $t_2 + s$ is $s_i-1$. Therefore $p$ divides $\binom{t_2+s}{s}$.
Hence, when we apply \eqref{4.4} to $\xi_{l(1,p^b),l} \xi_{l,j}$, only the
summands corresponding to $s=p^b$, $t=0$ and $s=0$, $t=p^b$ remain.

If $b=d+1$ a similar argument applies. Only this time, the condition
$t<p^{d+1}$ in \eqref{4.4}
implies that we are left only with the summand that corresponds to $s=p^b$,
$t=0$.
\end{proof}
\begin{lemma}
	\label{lemma4.3}
Given $j$,  $j'$ and $b=b(j)$
as above
we have:
\begin{enumerate}[(i)]
	\item if $b<d+1$ and $p_1^t\left( \xi_{l,j'} \right)\not=0$, then
		$p_1^t\left( \xi_{l,j} \right)$ and $p_1^t\left( \xi_{l,j'}
		\right)$ are linearly independent;
	\item if $b=d+1$, then $p_1^t\left( \xi_{l,j} \right)$ and $p_1^t\left(
		\xi_{l,j'} \right)$ are linearly dependent.
\end{enumerate}
\end{lemma}
\begin{proof}
	Suppose $b<d+1$. Using Lemma~\ref{lemma4.2}, we only have to make sure
	that $\xi_{l(1,p^b), j'}$ is different from $\xi_{l(1,p^b),j}$ and
	$\xi_{i,j}$. Note first that $j'\not=j\pi$ for any $\pi\in
	\Sigma_{\lambda(1,p^b)}$. In fact, $\pi$ of $\Sigma_{\lambda(1,p^b)}$ can move at most
	$p^b$
	$2$'s from the second row of $T^\lambda_j$ to its first row. But
	$T^\lambda_{j'}$ is obtained from $T^{\lambda}_j$ by moving exactly
	$p^{d+1}$  $2$'s from the second to the first row. As $p^{d+1}>p^b$ this
	can not be achieved by application of $\pi\in \Sigma_{\lambda(1,p^b)}$.
	Thus $\left( l\left( 1,p^b \right), j \right)\not\sim \left( l\left(
	1,p^b \right), j'
	\right)$. In a similar way, we see that $\left( l(1,p^b), j'
	\right)\not\sim \left( i,j \right)$, since no $\sigma$ satisfying
	$j\sigma = j'$ can move the $p^b$ $1$'s from the second row of
	$T^\lambda_i$ to the first $p^b$ positions of this row.

	Now consider the case $b=d+1$. In this case, the permutations that
	permute the $p^{d+1}$ $2$'s in the second row of $T^\lambda_j$ with the
	first $p^{d+1}$ $3$'s in the first row belong to
	$\Sigma_{\lambda(1,p^{d+1})}$. So $\xi_{l\left( 1,p^{d+1} \right), j'} =
	\xi_{l\left( 1,p^{d+1} \right),j}$.
Notice also that $b=d+1$ implies that
		$ t_3 = (p-1) + \dots + (p-1) p^d + c''_{d+1} p^{d+1} + \cdots.$
	Thus
	$t_3 - p^{d+1} + p^{d'} = (p-1) + \dots + (p-1)p^{d'-1} + c''_{d+1}
	p^{d+1} + \cdots.$
Hence
$	p_1^t\left( \xi_{l,j'} \right) = \binom{t_3}{p^{d+1}} \xi_{l\left(
	1,p^{d+1}
	\right),j'}  = \binom{t_3}{p^{d+1}} \xi_{l( 1,p^{d+1} ), j}.$
\end{proof}
At this point, we should remark that if $t_3\not=0$, then $j\not\in J\left(
\lambda\left( 1,p^b \right)
\right)$ for any $b$, since $T^{\lambda(1,p^b)}_j$ is not row semistandard. But
$\xi_{l(1,p^b),j}= \xi_{l(1,p^b),j\pi}$, for any $\pi$ in the row stabilizer of
$T^{\lambda(1,p^b)}$, and we can choose $\pi$ such that $j\pi\in J\left(
\lambda(1,p^b)
\right)$. So $\xi_{l(1,p^b),j}= \xi_{l(1,p^b),j\pi} \in B_2$. In the particular
case of $b=d+1$, we have $\xi_{l(1,p^{d+1}), j} = \xi_{l(1,p^{d+1}),j'}$.

The following is the main result of this section.
\begin{proposition}
	\label{prop4.4}
	Let $\ch\K= p$, $d$ a natural number,  and $\lambda = \left( \lambda_1,\lambda_2,\lambda_3
	\right)\in \Lambda(3,r)$ with $p^d\le \lambda_3<p^{d+1}$. Suppose
	$\lambda_2 = 2p^{d+1} -1$. Let $B_2$ be the basis~\eqref{1.0} of
	$P_1^t$. Then we obtain from $B_2$ a new basis $\hat{B}_2$ of $P_1^t$ in
	the following way: given $j\in J\left( \lambda \right)$, suppose the
	$p$-adic expansion of $a(j)$ is $\sum_{q=0}^{d+1} a_qp^q$. Then:
	\begin{enumerate}[(a)]
		\item If $a_q \not=p-1$, for some $q\le d$, we replace
			$\xi_{l(2,p^{m(j)}),j}$ by $p_1^t(\xi_{l,j})$ in
			$B_2$.
		\item If $a(j)= p^{d+1}-1$, let $t_2$ and
			$t_3$ be the number of $2$'s and $3$'s in the first row
			of $T^\lambda_j$, respectively.
			\begin{enumerate}[(i)]
				\item
			If the $p$-adic expansions of $t_2$ and $t_3$ are:
\begin{align*}
	t_2 & =(p-1) + \dots + (p-1)p^{b-1} + c'_b p^b + \dots \\
	t_3 &= (p-1) + \dots + (p-1) p^{b-1} + c''_b p^b + \dots\\& \not=
	(p-1) + \dots + (p-1) p^d + 0\cdot p^{d+1} + \dots
\end{align*}
and $c'_b \not= p-1$ or $c''_b\not=p-1$ for some $b<d+1$,
then in $B_2$, we replace $\xi_{l(1,p^b), j}$ and $\xi_{l(1,p^{d+1}), j'}$ by
$p_1^t\left( \xi_{l,j} \right)$ and $p_1^t(\xi_{l,j'})$, respectively.
\item If the $p$-adic expansions of $t_2$ and $t_3$ are:
\begin{align*}
	t_2 & =(p-1) + \dots + (p-1)p^{b-1} + c'_b p^b + \dots \\
	t_3 &= (p-1) + \dots + (p-1) p^d + 0\cdot p^{d+1} + \dots
\end{align*}
for some $b<d+1$  such that $c'_b \not=p-1$, then $p_1^t\left(
\xi_{l,j'} \right)= 0
$ and we replace $\xi_{l(1,p^b),j}$ by $p_1^t(\xi_{l,j})$ in $B_2$.
\item If the $p$-adic expansions of $t_2$ and $t_3$ are:
\begin{align*}
	t_2 & =(p-1) + \dots + (p-1)p^{d} + c'_{d+1} p^{d+1} + \dots \\
	t_3 &= (p-1) + \dots + (p-1) p^d + c''_{d+1}p^{d+1} + \dots
\end{align*}
with $c''_{d+1}\not=0$, then we replace $\xi_{l(1,p^{d+1}), j}$ by
$p_1^t(\xi_{l,j'})$ in $B_2$. In this case $\xi_{l(1,p^{d+1}), j} =
\xi_{l(1,p^{d+1}), j'}$ and $p_1^t(\xi_{l,j})$ is a multiple of
$p_1^t(\xi_{l,j'})$.
\item If the $p$-adic expansions of $t_2$ and $t_3$ are:
\begin{align*}
	t_2 & =(p-1) + \dots + (p-1)p^{d} + c'_{d+1} p^{d+1} + \dots \\
	t_3 &= (p-1) + \dots + (p-1) p^d + 0 \cdot p^{d+1} + \dots
\end{align*}
with $c'_{d+1}\not=p-1$, then $p_1^t(\xi_{l,j'}) = 0$ and we replace
$\xi_{l(1,p^{d+1}), j'}$ by $p_1^t(\xi_{l,j})$ in $B_2$
	\end{enumerate}
\item If $a(j) = 2p^{d+1}-1$ with $t_2<p^{d+1}$.
	\begin{enumerate}[(i)]
		\item If the $p$-adic expansion of $t_3$ is:
			\begin{equation*}
				t_3 = (p-1) + \dots + (p-1) p^{b-1} + c_b p^b +
				\dots
			\end{equation*}
			with $b\le d+1$ and $c_b\not= p-1$, then we replace
			$\xi_{l(1,p^b),j}$ by $p_1^t(\xi_{l,j})$ in $B_2$.
		\item If the $p$-adic expansion of $t_2$ is
			\begin{equation*}
				t_3 = (p-1) + \dots + (p-1)p^{d+1} + \dots,
			\end{equation*}
			then $p_1^t(\xi_{l,j}) = 0$.
	\end{enumerate}
	\end{enumerate}
\end{proposition}
\begin{proof}
Putting together Lemma~\ref{lemma4.1}, Lemma~\ref{lemma4.3}, and
the result obtained for $j$ not a critical element of $J\left( \lambda
\right)$, we obtain (a) and (b).

If $a(j) = 2p^{d+1}-1$ and $t_2 \ge p^{d+1}$, then there is $\tilde\jmath\in
J(\lambda)$ with
$a(\tilde\jmath)=p^{d+1}-1 $ such that $j=\tilde\jmath'$.
Thus this case is considered in (b). Now suppose that $t_2<  p^{d+1}$.
By the formula similar to~\eqref{4.7}, we get
\begin{equation}
	\label{t3}
	p_1^t(\xi_{l,j}) = \sum_{0\le d'\le d+1}
	\binom{t_3 +p^{d'}}{p^{d'}} \xi_{l(1,p^{d'}),j}.
\end{equation}
If the $p$-adic
expansion of $t_3$ is of the form $	t_3 = (p-1) + \dots + (p-1)p^{d+1} + \cdots,$
then~\eqref{t3} implies that $p_1^t(\xi_{l,j})=0$.
Otherwise, let $b$
be the first coefficient of the $p$-adic expansion of $t_3$ different from zero.
Then the coefficient of $\xi_{l(1,p^b),j}$ in \eqref{t3} is non-zero.
Thus we can replace $\xi_{l(1,p^b),j}$ by $p_1^t(\xi_{l,j})$ in $B_2$.
\end{proof}
To construct an Auslander-Reiten sequence ending with $\K_\lambda$,  we repeat the
procedure used in the previous cases. Let $\hat{B}^*_2$ be the basis of
$DP_1^t$ dual to $\hat{B}_2$. Define $U_\lambda$ as the $S(B^+)$-submodule of
$DP_1^t$ spanned by the elements of $\hat{B}^*_2$ which do not correspond to
$p_1^t(\xi_{l,j})$. If $E(\lambda)$, $f$, and $g$ are defined as in
Theorem~\ref{prop1}, then an Auslander-Reiten sequence ending with $\K_\lambda$
is	$0\to U_\lambda \stackrel{f}{\longrightarrow} E\left( \lambda
		\right)
		\stackrel{g}{\longrightarrow} \K_\lambda \to 0.$

	\section{The socle of $S^+(n,r)$}
	\label{soclesection}
	In the previous sections we studied the kernel of the map $p_1^t$.
	Since this kernel can be identified with  $\Hom_{S(B^+)}(\K_\lambda,
	S(B^+))$, it provides information on the socle of the Borel-Schur
	algebra $S(B^+)$. Namely,  $p_1^t$ is non-injective if and only if
	$\K_\lambda$ is in the socle of $S(B^+)$.
	In this section we collect some facts on the socle of $S(B^+)$.
	We will use the usual notation $\Lambda^+(n,r)$ for the subset of
	partitions in $\Lambda(n,r)$.

	We start with the following auxiliary result.
\begin{lemma}
	\label{propvanish1}
Suppose $\nu\in \Lambda(n,r)\setminus \Lambda^+(n,r)$ and let $M$ be an
$S(n,r)$-module.
	Then
	$\Hom_{S(B^+)}(\K_\nu, M) = 0$, where we consider $M$ as an
	$S(B^+)$-module by restriction.
\end{lemma}
\begin{proof}
Let $f\colon \K_\nu\to M$ be an $S(B^+)$-homomorphism and
$c\in \K_\nu$. Then $\xi_{ij}f(c) = f(\xi_{ij}c) = 0$
for all $\xi_{ij}\in
S(B^+)$ different from $\xi_\nu$. By \cite[Theorem~5.2]{green1} we have $
	S(n,r) = \sum_{\lambda \in \Lambda^+(n,r)} S(B^-) \xi_\lambda S(B^+),$
where $S(B^-)$ denotes the lower Borel subalgebra of the Schur algebra
$S(n,r)$.
Since $\nu\not\in \Lambda^+(n,r)$, we get that $S(n,r) f(c)=0$. This shows that
$f(c)=0$ for all $c\in \K_\nu$ and thus $f$ is the zero map.
\end{proof}
As a simple consequence we get:
\begin{proposition}
	\label{simple}
	Let $\nu\in \Lambda(n,r)\setminus \Lambda^+(n,r)$. Then
	$\Hom_{S(B^+)}(\K_\nu, S(B^+))=0$.
\end{proposition}
\begin{proof}
	The embedding $S(B^+)\hookrightarrow S(n,r)$ induces the injective map
	\begin{equation*}
		\Hom_{S(B^+)}(\K_\nu, S(B^+)) \to \Hom_{S(B^+)}(\K_\nu, S(n,r)).
	\end{equation*}
		Now Lemma~\ref{propvanish1} implies that the vector space $\Hom_{S(B^+)}(\K_\nu,
	S(n,r))$ is trivial. Thus also $\Hom_{S(B^+)}(\K_\nu, S(B^+))$ vanishes.
\end{proof}
Combining the results of the previous sections and Proposition~\ref{simple}, we
get the following theorem.
\begin{theorem}
	\label{socle}
\begin{enumerate}
	\item The module $\K_{(r,0,\dots,0)}$ is a direct summand of the socle
		of $S(B^+)$ independently of $\ch\K$.
	\item Suppose $\ch\K=0$ and $n=2$. Then the socle of $S(B^+)$
		is a direct sum of several copies of $\K_{(r,0)}$.
	\item Suppose $\ch\K=p$ and $n=2$. Then
		$\K_\lambda$ is a direct summand of the socle of $S(B^+)$ if and
		only if  $\lambda=(r,0)$ or $\lambda$ is a partition satisfying
		\begin{equation*}
			\lambda_1 \ge p^{\lfloor \log_p \lambda_2\rfloor+1} -1.
		\end{equation*}
	\item Suppose $n\ge 3$ and $\ch \K =0$. Then all the composition factors of the socle of
		 $S(B^+)$ are of the form $\K_\lambda$ with $\lambda$ a
		 partition such that $\lambda_n=0$.
	 \item Suppose $n\ge 3$ and $\ch \K =p$. Then all the composition
		 factors of the socle of $S(B^+)$ are of the form $\K_\lambda$
		 with $\lambda$ a partition such that either $\lambda_n=0$ or
		 \begin{equation*}
			 \lambda_{n-1} \ge p^{\lfloor \log_p \lambda_n \rfloor
			 +1} -1.
		 \end{equation*}
	 \item Suppose $n=3$, $\ch\K=p\ge 3$, and $\lambda = \left(
		 \lambda_1,\lambda_2,\lambda_3
		 \right)$ satisfies
		 \begin{equation}
			 \label{estimates}
			 \lambda_1 \ge p^{\lfloor \log_p \lambda_3 \rfloor
			 +2}-1,\ \lambda_2 = 2p^{\lfloor \log_p \lambda_3\rfloor +1}-1.
		 \end{equation}
		 Then the module $\K_\lambda$ has a non-zero multiplicity in the
		 socle of $S(B^+)$.
\end{enumerate}
\end{theorem}
\begin{proof}
	The simple module $\K_{(r,0,\dots,0)}$ is isomorphic to its projective cover
	$S(B^+)\xi_{(r,0,\dots,0)}$.
	Thus $\Hom_{S(B^+)}(\K_{(r,0,\dots,0)}, S(B^+)) \cong
	\xi_{(r,0,\dots,0)} S(B^+)$ is non trivial.
	This shows that $\K_{(r,0,\dots,0)}$ is a direct summand of the socle of
	$S(B^+)$.
	
	Now the claims (2)-(5) follow from
	Propositions~\ref{lemma2(2)},~\ref{lemma3}, and~\ref{lemma4}.

To prove (6), we denote $\lfloor \log_p \lambda_3 \rfloor$ by $d$. Then
$p^d \le \lambda_3<p^{d+1}$. Since $\lambda$ satisfies~\eqref{estimates}
there is $j\in J(\lambda)$ such that $a(j) =\lambda_2=  2p^{d+1}-1$ and
$t_3 = t_3(j) = p^{d+2}-1$.
Then by Proposition~\ref{prop4.4}(c)(ii) we get
$p_1^t(\xi_{l,j}) =0$. This shows that $\ker(p_1^t)$ is non-trivial, and
therefore there is an embedding of $\K_\lambda$ into the socle of $S(B^+)$.
\end{proof}

\section{ Functors between different algebras}\label{functor}
A Borel-Schur algebra is \emph{triangular}, that is, its quiver does not have oriented cycles. Furthermore,  a Borel-Schur algebra $S(B^+,m,r)$ is a non-unital subalgebra
of a Borel-Schur algebra $S(B^+,n,r)$ for $m<n$. This gives rise to functors between their module categories, and one would like to understand when an  Auslander-Reiten sequence
of a module for the smaller algebra lifts to an Auslander-Reiten sequence over the larger algebra.

We study this question in a slightly more general setting.
Assume $\Lambda^* \subset \Lambda(n,r)$ is a coideal with respect to the dominance order. That is, if $\lambda \in \Lambda^*$ and ${\lambda\ledom \mu}$  then
$\mu \in \Lambda^*$. For example, if $m<n$  then $$\Lambda^*(m, r):= \{ \alpha \in \Lambda(n, r)\,|\, \alpha = (\alpha_1, \ldots, \alpha_m, 0, \ldots , 0)\}$$ is
a coideal in $\Lambda(n,r)$.   Let
$$e=e_{\Lambda^*} := \sum_{\lambda \in \Lambda^*} \xi_{\lambda}.
$$
This is an idempotent  in $A:= S(B^+,n,r)$. Then $eAe$ can be regarded as a non-unital subalgebra of $A$, see also \cite[\S~6.5]{green}. This
algebra has simple modules precisely the $\K_{\lambda}$ with $\lambda \in \Lambda^*$.
In case $\Lambda^* = \Lambda^*(m, r)$ we have
\begin{equation}
eS(B^+,n,r)e \cong S(B^+,m,r).\label{new2}
\end{equation}

We have an exact functor
\begin{align*}
        F\colon A\mbox{-mod} & \rightarrow eAe\mbox{-mod}\\
        V & \mapsto e V\\
        (\theta\colon V\to V') &\mapsto \theta|_{eV}.
\end{align*}
The functor $F$ has a left adjoint $
        G= Ae\otimes_{eAe} \colon eAe\mbox{-mod}  \rightarrow A\mbox{-mod}.$
       In
our case $G$ is   the identity functor,  since $(1-e)Ae=0$.
We get as an easy consequence:
\begin{proposition} With the above notation, the following hold:
\begin{enumerate}[(i)]
\item $FG(M)=M$ for all $M\in eAe\mbox{-mod}$;
\item $G$ is exact;
\item$G$ preserves indecomposable modules;
\item$G$ preserves simple modules.
\end{enumerate}
\end{proposition}

We will use the following notation. If $\lambda\in \Lambda(n,r)$ belongs to the subset $\Lambda^*$ we write $\bar{\lambda}$ if we view it as a weight of the algebra
$eAe$, and we write $\K_{\bar{\lambda}}$ for the simple $eAe$-module labelled by $\lambda$.  Then $G(\K_{\bar{\lambda}}) = \K_{\lambda}$.
Consider the Auslander-Reiten sequence in $eAe$-mod
$$0\to \tau\K_{\bar{\lambda}}\to E(\bar{\lambda}) \to \K_{\bar{\lambda}}\to 0.
$$
By applying  the  functor $G$ we obtain an exact sequence
\begin{equation}\label{New}
0\to G(\tau\K_{\bar{\lambda}})\to G(E(\bar{\lambda}))\to
\K_{\lambda}\to 0.
\end{equation}

\begin{proposition}\label{ARiff}   The sequence~(\ref{New}) is an Auslander-Reiten sequence if and only if
$G(\tau \K_{\bar{\lambda}}) \cong  \tau \K_{\lambda}$.
\end{proposition}
\begin{proof}One direction is clear. For the converse, assume $G(\tau \K_{\bar{\lambda}}) \cong   \tau\K_{\lambda}$.  The exact sequence~(\ref{New}) is non-split since,
if we apply $F$, we get the Auslander-Reiten sequence in $eAe$-mod.  By~Remark \ref{be},  any non-split exact sequence with end terms
$\tau \K_{\lambda}$ and $\K_{\lambda}$ is the Auslander-Reiten sequence. This proves the claim.
\end{proof}

We give now an example where the sequence~(\ref{New}) is not an Auslander-Reiten sequence.

\begin{example} Suppose char$(\K)\neq 2$.   Assume $\Lambda^* = \Lambda(2, 3)$ viewed as a subset of $\Lambda(3, 3)$, so that $eAe\cong S(B^+,2,3)$ and $A=S(B^+,3,3)$. Consider the weight
$\lambda = (1, 2,0)$.
From the construction, we have the exact sequence
$$0\to \tau \K_{\lambda} \to DP_1^t\stackrel{Dp_1^t}\to D P_0^t.
$$ In this case, $DP_1^t$ is the injective $I_{(2,1,0)}$ with socle $\K_{(2,1,0)}$  and $DP_0^t$ is the injective $I_{(1,2,0)}$.

There is a uniserial module $U$ of dimension $2$ with top isomorphic to $\K_{(2,0,1)}$ and socle isomorphic to $\K_{(2,1,0)}$. This is a submodule
of $I_{(2,1,0)}$ and it is contained in the kernel of $Dp_1^t$, as the socle of $I_{(1,2,0)}$ is not a composition factor of $U$.
That is, $U$ is a submodule of $\tau\K_{\lambda}$ and we see that $\tau\K_{\lambda}$ is not of the form $G(M)$ for any $M$. So by the Proposition~\ref{ARiff},
(\ref{New}) cannot be an Auslander-Reiten sequence.
\end{example}

\section{Finite type classification }\label{finitetype}

In this section we will determine precisely which Borel-Schur algebras are of finite type.
We start by summarising general results which can be used to identify representation type.
 For the moment, assume $A$ is some finite-dimensional algebra.
Recall that $A$ has finite type if there are only finitely many indecomposable $A$-modules up to isomorphism.
Otherwise, $A$ has infinite type. If $A=\K\cQ$, where $\cQ$ is a quiver with no oriented cycles, then
$A$ has finite type if and only if $\cQ$ is a disjoint union of Dynkin quivers,
of types ADE. This is  Gabriel's Theorem (see~\cite{gabriel_Dynkin}).
In some cases, the algebra will be of the form $\K\cQ/\cI$ where $\cI$ is generated by a single commutativity relation.
We will use a result by Ringel (see \cite{Ri}) to identify infinite type.
Furthermore, the following is a general reduction method
(see for example Lemma~1 in \cite{bongartz}).

\begin{lemma}\label{eAe}
 Assume $e$ is an idempotent of $A$. If $eAe$ has infinite type then $A$ has infinite type.
 \end{lemma}

We will combine this result with the idea of regular covering of quivers.
Let $\cQ$ be a quiver and $G$ a group acting freely on the right of $\cQ$.  Then one has the
quotient quiver
$\cQ' := \left.\raisebox{0.3ex}{$\cQ$}\middle/\!\raisebox{-0.3ex}{$G$}\right. $ and the
canonical projection
$\phi \colon \cQ \to
\cQ'$.
In that situation one says that $\cQ$ is  a \emph{regular covering} of
$\cQ'$.

The action of $G$ on $\cQ$ induces an action of $G$ on the path category
$P(\cQ)$ of the quiver $\cQ$. It is obvious that
$\left.\raisebox{0.3ex}{$P(\cQ)$}\middle/\!\raisebox{-0.3ex}{$G$}\right. $ is
isomorphic to the path category $P(\cQ')$.
In particular, following Gabriel~\cite{gabriel_inventiones}, we can define
the pushdown functor $\phi_* \colon P(\cQ)\mbox{-mod} \to P(\cQ')\mbox{-mod}$.
In our case the definition takes the following form.
 Let $V$ be a representation
of $\cQ$, then for any vertex $xG\in \cQ'$
\begin{equation*}
	\phi_* (V)_{xG} := \bigoplus_{g \in G} V_{xg};
\end{equation*}
and for every arrow $x \xrightarrow{a} y$ in $\cQ$, we define the map
\begin{equation*}
\phi_* (V)_{aG}\colon \phi_*(V)_{xG}\to \phi_*(V)_{yG}
\end{equation*}
	to be the matrix with zeros in positions $(g_1,g_2)$ if
$g_1\not= g_2$ and $ag$ at position $(g,g)$.

From Lemma~3.5 in \cite{gabriel_LNM}
we get
\begin{theorem}
	\label{covering}
	Let $\cQ$ be a quiver and $G$ a group acting freely on $\cQ$.
Suppose that
$V$ is a finite dimensional indecomposable representation of $\cQ$ such that
$g_* V \not\cong V$, for every $g\in G$,  $g\neq 1_G$.
Then $\phi_* V$ is indecomposable. Moreover, if $W\not\cong V$ is a representation of
$\cQ$ such that $\phi_* W \cong \phi_* V$, then there is $g\in G$, $g\not=1_G$, such
that $g_* V \cong W$.
\end{theorem}

 Returning to Borel-Schur algebras, we have already seen in~\eqref{new2}  that $S(B^+,m,r)$ $ \cong $ $eS(B^+,n,r)e,$ for some idempotent $e\in S(B^+,n,r),$ whenever $n > m$, and
 we will apply Lemma~\ref{eAe} in this case. It also applies to relate $S(B^+,2,r)$ with $S(B^+,2,r')$ for $r'\geq r$. Namely we have:

 \begin{lemma}\label{yu} There is an idempotent $e$ of  $S(B^+,2,r+1)$ such that $$eS(B^+,2,r+1)e\cong S(B^+,2,r).$$
 \end{lemma}
\begin{proof}
 The usual basis of $S(B^+,2,r)$  can be parametrized as $\xi_{\nu, \mu}$ where $\nu, \mu\in \Lambda(2,r)$ and  $\mu\ledom \nu$, see \cite{Y}.  If $\lambda\in \Lambda(2,r)$, then
 set $\hat{\lambda} = (\lambda_1+1, \lambda_2)\in \Lambda(2, r+1)$. Let
 $$e:= \sum_{\lambda\in \Lambda(2,r)} \xi_{\hat{\lambda}}.$$
 Then we claim that $eS(B^+,2,r+1)e\cong S(B^+,2,r)$. Namely the  linear map, defined on the basis by
 $\xi_{\mu, \lambda}\mapsto \xi_{\hat{\mu}, \hat{\lambda}} $, is bijective and, by Lemma 22 of \cite{Y}, it is an algebra map.
\end{proof}

Some of the Borel-Schur algebras  $S(B^+,2,r)$ are special biserial.
\begin{definition} An algebra of the form $A=K\cQ/\cI$
is special biserial if \\
(i) For each vertex $i$ of $\cQ$, there are at most two arrows starting at $i$ and at most two arrows ending at $i$.\\
(ii) For each arrow $\alpha$ of $\cQ$ there is at most one arrow $\beta$ and one arrow $\gamma$ such that $\alpha\beta$ and $\gamma\alpha$ do not belong to
$\cI$.
\end{definition}

The indecomposable modules of such algebras are classified in \cite{wald}. The
ones which are not simple and not projective can be parametrised by admissible words in the arrows and their inverses modulo suitable equivalence relations. A word is admissible if no linear subword occurs
in some relation in $\mathcal{I}$. In our cases there will only be finitely many admissible words
so that we can deduce finite type.

 We can now state the classification of finite type for Borel-Schur algebras.

 \begin{theorem}
	 \label{reptype}
	 \begin{enumerate}
 \item The algebra $S(B^+,2,r)$ has finite type if and only if one of the following holds:
\begin{enumerate}
\item[(i)]  ${\rm char}(K)=0$;
 \item[(ii)] ${\rm char}(K)=p\geq 5$ and $r\leq p$;
 \item[(iii)] ${\rm char}(K)=3$ and $r\leq 4$;
 \item[(iv)] ${\rm char}(K)=2$ and $r\leq 3$.
 \end{enumerate}
\item For $n\geq 3$, the algebra $S(B^+,n,r)$ has finite type if and only if $r=1$.
\end{enumerate} \end{theorem}
\begin{proof}  (1)  We show first that the algebras listed above have finite type.

\emph{ (a)}  If char$(\K)=0$, or if char$(\K)=p$ and $r<p$ then $S(B^+,2,r)$ is isomorphic to $\K\cQ$ where $\cQ$ is the Dynkin quiver of type $A_{r+1}$ with linear orientation.
This follows from Lemma~22 and Proposition~24 in \cite{Y}. By Gabriel's Theorem, the algebra has finite type.

 From now we assume that $\K$ has characteristic $p$.
 We label the quiver of $S(B^+,2,r)$ by $0, 1, \ldots, r$ writing $i$ for the composition $(r-i, i)$.

\emph{ (b)} If $r=p$ then $ S(B^+,2,p)$ has finite type:  the algebra has the form
\begin{equation*}
	\xygraph{
		!{<0ex,0ex>;<8ex,0ex>:<0ex,0.3ex>::}
		!{(0,6) } :_(-0.1){p}_(1){p-1}^{\alpha_{p-1}}[l]
		:_(1){p-2}^{\alpha_{p-2}}[l]:@{.}_(1){2}[l]:_(1){1}^{\alpha_{1}}[l]:_(1.1){0}^{\alpha_{0}}[l]:@{<-}@/^7ex/[rrrrr]^{\beta}
	}
\end{equation*}
with $\alpha_i: i+1\to i$ and
$\beta:p\to 0$,
 where  the only relation is that the product of all the $\alpha_i$ is zero.
 This is (trivially)  special biserial.  Any admissible word is a subword of
 $\alpha_1\alpha_2\ldots \alpha_{p-1}\beta^{-1}\alpha_0\alpha_1\ldots \alpha_{p-2}$ (or its inverse). Hence there are only finitely many admissible words and
 the algebra has finite type.

\emph{ (c)} The algebra $S(B^+,2,3)$ for $p=2$ also is special biserial.
 If $\alpha_i$ is the arrow $i+1\to i$ and $\beta_i$ is the arrow $i+2\to i$, then the admissible words are those which do not
 have a subword $\alpha_i\alpha_{i+1}$ or
 $\beta_1\alpha_2$ or $\alpha_0\beta_2$. The length of admissible words is bounded (by 6). So the algebra has finite type.

\begin{figure}
	\includegraphics[trim=4cm 4cm 6cm 4cm]{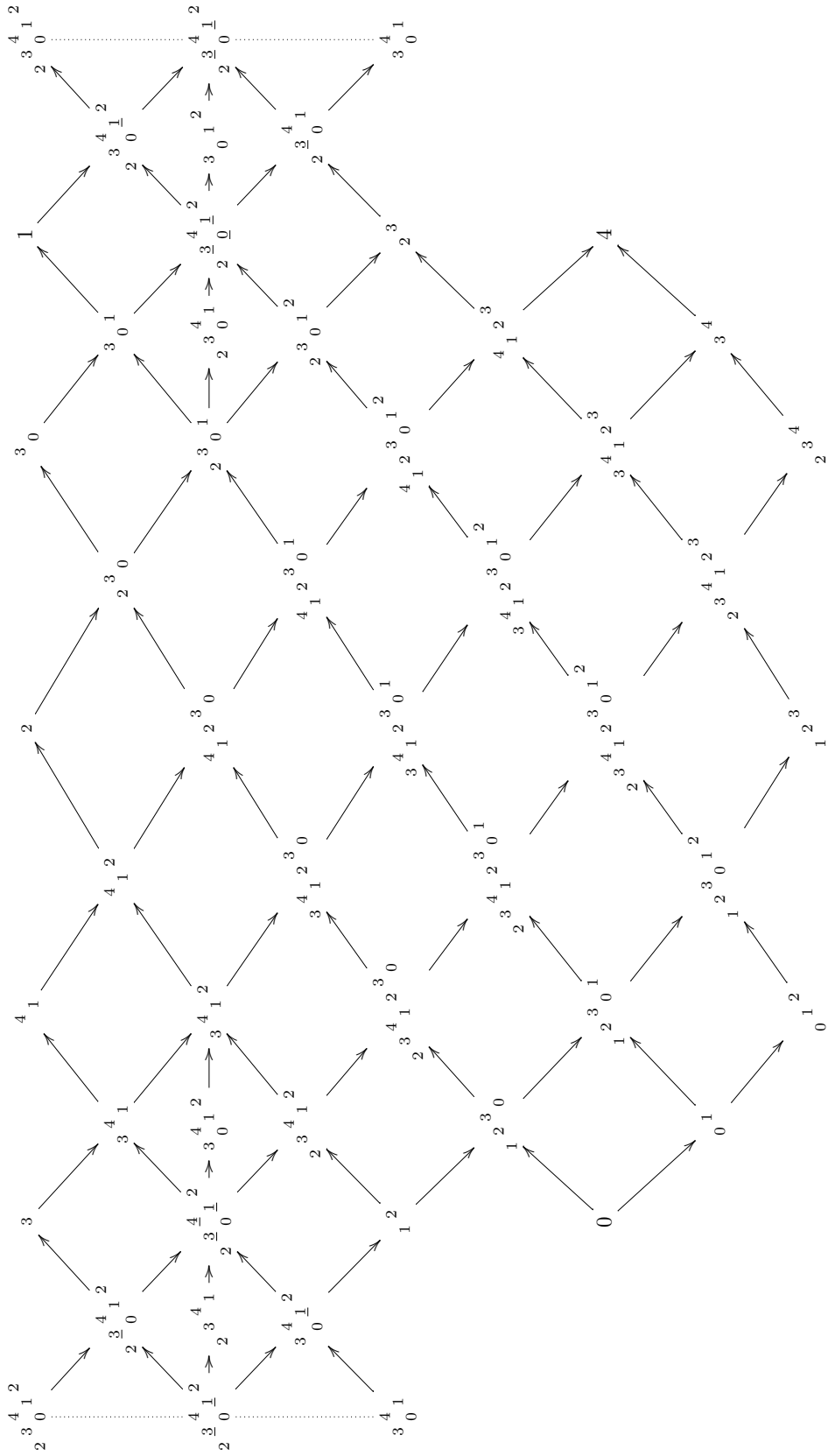}
	\caption{Auslander-Reiten quiver for $S(B^+,2,4)$, $p=3$.}
	\label{fig:AR}
\end{figure}

 \emph{(d)} The algebra $S(B^+,2,4)$ has finite type for $p=3$.  To prove this,
 we calculate   Auslander-Reiten sequences, and we find  that the
 Auslander-Reiten quiver has a finite component, a drawing of which we include.
 Auslander's Theorem (see for example \cite{Ri})  states that if the
 Auslander-Reiten quiver of an indecomposable algebra has a finite component then
 this component contains precisely all indecomposable modules, and hence the algebra has finite type.
  Recall that any  $S(B^+,n,r)$ is indecomposable.
 Hence $S(B^+,2,4)$ has finite type.

 Now we find four classes of algebras and show they have infinite type.

\emph{ (e) } Consider $S(B^+,2,p+1)$ for $p\geq 7$.
Define
\begin{equation*}
	e = \xi_{(p+1,0)} + \xi_{(p,1)} + \xi_{(p-1,2)} + \xi_{(p-2,3)}+
	\xi_{(3,p-2)}  + \xi_{(2,p-1)} +
	\xi_{(1,p)} + \xi_{(0,p+1)}.
\end{equation*}
Then the algebra $eS(B^+,2,p+1)e$ is the quiver algebra of the quiver
\begin{equation}
	\label{Ri32}
	\xygraph{
		!{<0ex,0ex>;<4ex,0ex>:<0ex,4ex>::}
		[rrrrrr]
		:^(0){3}^(1){2}[l]:^(1){1}[l]="r":^(1.2){0}[l]="b"
		[ur]="t":_(1){p}[l]="l":_(1){p-1}[l]:_(1.4){\phantom{p-}p-2}[l] "t":|(-0.3){\phantom{1+}p+1}"r" "l":"b"
	}	
\end{equation}
  with commuting square and no other relation.
  The quiver~\eqref{Ri32} is number $32$ in  Ringel's
  list~\cite{Ri} and the algebra  has infinite type.

  \emph{(f)} Consider $S(B^+,2,6)$ for $p=5$. It is the quiver algebra of
  the quiver
\begin{equation}
	\label{quiver265}
		\xygraph{
!{<0cm,0cm>;<6ex,0ex>:<0cm,6ex>::}
[rrrrr]:^(1)5_-{\alpha_5}^(0)6[l]:^-{\alpha_4}_(1)4[l]:^-{\alpha_3}_(1)3[l]:^-{\alpha_2}_(1)2[l]:_-{\alpha_1}^(1)1[l]:^-{\alpha_0}^(1)0[l]
[rrrrrr]:@/^6ex/^{\beta_1}[lllll]
[rrrr]:@/_6ex/_{\beta_0}[lllll]
}
\end{equation}
with the relations
\begin{equation}
	\label{rel265}
	\begin{aligned}
\alpha_0	\alpha_1	\alpha_2 \alpha_3 \alpha_4 & = 0, &
		\alpha_1\alpha_2 \alpha_3 \alpha_4 \alpha_5 & = 0, &
		\alpha_0 \beta_1 & = \beta_0 \alpha_5.
	\end{aligned}
\end{equation}
Let us consider the quiver
\begin{equation}
	\label{cover265}
	\xygraph{
!{<0cm,0cm>;<0.7cm,0cm>:<0cm,0.7cm>::}
[drrrrrrrrrrr]
:_(0){3''}_(1){2''}[l]:_(0.8){1''}[l]:^(1.0){0''}[dl]:@{<-}[ul]:_(0){5''}_(1){4''}[l]:_(1){3'}[l]:_(1){2'}[l]:_(0.8){1'}[l]:^(1.0){0'}[ld]:@{<-}[ul]:_(0){5'}_(1){4'}[l]:@{-}[d]:@{-}@/_0.3cm/[dr]:@{-}
[rrrrrrrrr]:@{-}@/_0.3cm/[ur]:[u]
[ll]:@{<-}_(1.3){6''}[ul]:[dl]
[llll]:@{<-}_(1.3){6'}[ul]:[dl]
}
\end{equation}
with an action of the symmetric group $\Sigma_2$ given by
interchanging $'$ with $''$.
Then the quiver \eqref{quiver265} is a quotient of \eqref{cover265} under this
action.
The full subquiver of~\eqref{cover265} spanned by the vertices
$3'$, and $i''$ with $0\le i\le 6$ is isomorphic to the quiver~\eqref{Ri32}.
Therefore there is an infinite set $\mathcal{J}$ of pairwise non-isomorphic indecomposable
representations $V$ of the quiver~\eqref{cover265} such that $V_{i'} = 0$ for $i\not=3$,
and the square
\begin{equation*}
	\xymatrix@R0.5cm@C0.5cm{
&	V_{6''} \ar[rd]\ar[ld]\\
V_{5''} \ar[rd]&& V_{1''}\ar[ld] \\
& V_{0''}
	}
\end{equation*}
commutes.
It is easy to check that for every such $V$ the representation $\phi_*
(V)$ of~\eqref{quiver265} satisfies relations~\eqref{rel265}.
Moreover, $\phi_*(V)$ is indecomposable.
In fact, suppose $\phi_*(V)$ is decomposable. By Theorem~\ref{covering}, we
have that $V$ is
invariant under the action of~$\Sigma_2$. This implies that $V_{i''} =0$ unless
$i=3$, and $V_{3'} \cong V_{3''}$. But this is impossible, since $V$ is
indecomposable.

Further, if $V$ and $W\in \mathcal{J}$ are two different representations of~\eqref{cover265} such that $\phi_*(V)\cong \phi_*(W)$, then by Theorem~\ref{covering}, we have $V\cong \sigma^* W$, where $\sigma = (12)\in
\Sigma_2$. This implies that $V_{i''} \cong W_{i'} =0$ if $i\not=3$. Therefore
either $V$ is isomorphic to the  simple module $S_{3'}$ or to  the simple module $S_{3''}$. Let
$\mathcal{J}' = \mathcal{J}\setminus \left\{ S_{3'}, S_{3''} \right\}$. Then
$\mathcal{J}'$ is infinite and $\left\{ \phi_* (V) \middle| V \in
\mathcal{J}'\right\}$ is the set of pairwise nonisomorphic indecomposable
representations of $S(B^+,2,6)$ over a field of charcteristic $5$. This shows
that $S(B^+,2,6)$ is of infinite type for $p=5$.

 \emph{(g) } Consider $S(B^+,2,5)$ for $p=3$.
 This is the quiver algebra of
 the quiver
 \begin{equation}
	 \label{quiver253}
	 \xygraph{
!{<0cm,0cm>;<1cm,0cm>:<0cm,1cm>::}
[rrrrr]:^{\alpha_4}_(0)5_(1)4[l]:^{\alpha_3}_(1)3[l]:^{\alpha_2}_(1)2[l]:^{\alpha_1}_(1)1[l]:^{\alpha_0}_(1)0[l]
[rrrrr]:@/_5ex/_{\beta_2}[lll]
[rr]:@/^5ex/^{\beta_1}[lll]
[rr]:@/_5ex/_{\beta_0}[lll]
	 }
 \end{equation}
 with relations
 \begin{align}
	 \label{rel253}
\alpha_0 \alpha_1 \alpha_2 & = 0 ,& \alpha_1\alpha_2 \alpha_3 & = 0, & \alpha_2
\alpha_3 \alpha_4 & =0, & \alpha_0 \beta_1 &= \beta_0 \alpha_3 ,& \alpha_1
\beta_2 & = \beta_1 \alpha_4.
 \end{align}
 Let us consider the quiver
 \begin{equation}
	 \label{cover253}
	 \xygraph{
	 !{<0cm,0cm>;<0.7cm,0cm>:<0cm,0.7cm>::}
[drrrrrr]
:_(0){5''}_(1){4''}[l]:_(1){3''}[l]:^(1){2'}[l]:^(1){1'}[l]:^(1){0'}[l]
:@{<-}[u]:@{-}@/_0.3cm/[ld]:@{-}[d]:@{-}@/_0.3cm/[dr]:@{-}[rrrrr]:@(r,r)[u]
:^(0){2''}^(1){1''}[l]:^(1){0''}[l]:@{<-}[u]
[rr]:[d] [lu]:[d]
[ull]:@{<-}[u]:_(0){5'}_(1){4'}[l]:_(1){3'}[l] [r]:[d]
	 }
 \end{equation}
 with an action of $\Sigma_2$ given
  by interchanging $'$ and $''$.
  Then \eqref{quiver253} is the quotient of \eqref{cover253} under this action.
 Let $\mathcal{J}$ be the set of isomorphism classes of finite dimensional
indecomposable representations $V$ of \eqref{cover253}, such that
\begin{align*}
	V_{0'} & = 0 , & V_{1'} & = 0, &  V_{5''} & = 0,
\end{align*}
the map from $V_{3'}$ to $V_{2''}$ is zero, and
the square
\begin{equation*}
	\xymatrix{V_{3''} \ar[d] & V_{4''}\ar[l] \ar[d]\\ V_{0''} & V_{1''}
	\ar[l] }
\end{equation*}
commutes.
Then the elements of $\mathcal{J}$ can be identified with the isomorphism
classes of finite
dimensional indecomposable representations of the quiver~33 in  Ringel's
list~\cite{Ri}. Thus $\mathcal{J}$ is infinite.

Now, for every isomorphism class $[V]$ in $\mathcal{J}$, one can check that
$\phi_* V$ satisfies relations \eqref{rel253} and therefore can be considered as
a representation of $S(B^+,2,5)$ over a field of characteristic $3$.
Theorem~\ref{covering} implies that $\phi_* V$ is indecomposable. In fact,
if this is not the case, then $\sigma_* V \cong V$ and thus $V_{0''} =
V_{1''} = V_{5'} = 0$. Moreover, the map from $V_{3''}$ to $V_{2'}$ is zero.
Hence $V$ is the direct sum of subrepresentations $V'$ and $V''$
defined by
\begin{align}
	\label{v}
	V'_{i'} & = V_{i'} , & V'_{i''} & = 0; & V''_{i'} &= 0, & V''_{i''} &=
	V_{i''},
\end{align}
where $0\le i \le 5$.
Since $V$ is indecomposable, we get that either $V'=0$ or $V''=0$. But
$\sigma_*V \cong V$ implies that $\sigma_* V'\cong V''$ and $\sigma_* V'' \cong
V'$. Thus in both cases,
we get that $V=0$.

Further, for every $[V]\in \mathcal{J}$ there is at most one $[W]\in
\mathcal{J}$ different from $[V]$ such that $[\phi_* V] =[\phi_* W]$. In fact, by
Theorem~\ref{covering}, we get that $V \cong \sigma_* W$, where $\sigma =
(12) \in \Sigma_2$.
Therefore the set
\begin{equation*}
	\left\{\, [\phi_*V] \,\middle|\,  [V] \in \mathcal{J} \right\}
\end{equation*}
is infinite, and we get that $S(B^+,2,5)$ for $p=3$ has infinitely many pairwise
non-isomorphic indecomposable representations.

\emph{ (h)} Consider  $S(B^+,2,4)$ for $p=2$.
It can be considered as a quiver algebra with the quiver
\begin{equation}
	\label{quiver242}
	\xygraph{
	!{<0cm,0cm>;<1cm,0cm>:<0cm,1cm>::}
	[rrrr]:_{\alpha_3}_(-0.2){4}_(1)3[l]:^{\alpha_2}^(1)2[l]:^{\alpha_1}_(1)1[l]:_{\alpha_0}_(1.2)0[l]
	:@{<-}@/^4ex/[rr]^{\beta_0}:@{<-}@/^4ex/[rr]^{\beta_2}:@/^8ex/[llll]^{\gamma}
	[r]:@{<-}@/_4ex/[rr]_{\beta_1}
	}
\end{equation}
and relations\begin{align}
	\label{rel242}
\alpha_0 \alpha_1 & = 0,& \alpha_1 \alpha_2 & =0, & \alpha_2 \alpha_3 &=0, &
\alpha_0 \beta_1 & = \beta_0 \alpha_2, & \alpha_1 \beta_2 &= \beta_1 \alpha_3 ,
& \beta_0 \beta_2 & = 0.
\end{align}
Let us consider the quiver
\begin{equation}
	\label{cover242}
	\xygraph{
	 !{<0cm,0cm>;<0.7cm,0cm>:<0cm,0.7cm>::}
[ddrrrrr]
:^(0){4''}_(0.9){3''}[l]:_(0.7){2''}[l]:^(1){1'}[l]:^(1){0'}[l]:@{<-}[u]
:@{-}@/_0.3cm/[ld]:@{-}[d]:@{-}@/_0.3cm/[rd]:@{-}[rrr]:@(r,r)[u]
:^(0){1''}^(1){0''}[l]:@{<-}[u] [r]:[d] [ru]:[dll]
[urr]:@{-}[u]:@{-}@/_0.3cm/[lu]:@{-}[ll]:@/_0.3cm/[ld]_(0.8){2'}
[rr]:[d] [u]:_(0){4'}_(1){3'}[l]:[l] [r]:[d] [ur]:[dll]
	}
\end{equation}
with an
action of $\Sigma_2$
 given by interchanging $'$ and $''$.
 Then \eqref{quiver242} is the quotient of \eqref{cover242} under this action.
 Let $\mathcal{J}$ be the set of
isomorphism classes of indecomposable representations $V$ of \eqref{cover242} such that $V_{i'} =0$ for $i\not=2$ and the
maps from
$V_{3''}$ to $V_{1''}$, from $V_{3''}$ to $V_{2''}$  and from $V_{2'}$ to $V_{1''}$ are zero maps. Then $\mathcal{J}$ can be
identified with the set of isomorphism classes of indecomposable representations
of the quiver
\begin{equation*}
	\xygraph{
		!{<0cm,0cm>;<0.9cm,0cm>:<0cm,0.5cm>::}
		[rrr]:_(-0.2){2''}_(0.9){0''}[ld]:@{<-}_(1.2){1''}[rd]
		[lu]:@{<-}_(0.8){4''}[l]:_(1.2){2'}[lu] [rd]:^(1.2){3''}[ld]
	}
\end{equation*}
of type $\widetilde{D}_5$.
Thus $\mathcal{J}$ is inifinite. One can check that for every
$[V]\in \mathcal{J}$ the representation $\phi_* V$ of \eqref{quiver242} satisfies
the relations \eqref{rel242} and thus can be considered as a representation of
$S(B^+,2,4)$ over a field of characteristic $2$. Using Theorem~\ref{covering}
one can verify as above that $\phi_* V$ is indecomposable. Moreover, if
$\left[ V \right]$ and $\left[ W \right]\in \mathcal {J}$  are different and
$[\phi_* V] =[ \phi_*W]$ then $[V ] = [\sigma_* W]$.
Therefore the set $\left\{\,
[\phi_* V] \,\middle|\,  [V]\in \mathcal{J} \right\}$
is infinite. Hence $S(B^+,2,4)$ defined over a field of characteristic
$2$ has
infinitely many  pairwise non-isomorphic indecomposable representations.
 This completes the proof of part \emph{(1)} of Theorem~\ref{reptype}.

 \emph{ (2) } Now we study $S(B^+,3,r)$ for $r\geq 2$. The quiver of
 $S(B^+,3,r)$ contains
 a full subquiver $\cQ'$ given by
 \begin{equation*}
	 \xygraph{
	 !{<0ex,0ex>;<9ex,0ex>:<0ex,9ex>::}
	 [r]:@{<-}_(-0.2){
	 (r-1,0,1)}_(1.2){(r-2,1,1)}[l]:[d]:_(-0.2){(r-2,2,0)}_(1.2){(r-1,1,0)}[r]:@{<-}[u]
	 }	
 \end{equation*}
 Let $e$ be the idempotent $e=\xi_{(r-2,1,1)}+\xi_{(r-1,0,1)} + \xi_{(r-2,2,0)} + \xi_{(r-1,1,0)}$. Then
 $eS(B^+,3,r)e$ is isomorphic to $K\cQ'$ and this is of infinite type, by Gabriel's Theorem.
 Hence $S(B^+,3,r)$ is of infinite type, by  Lemma~\ref{eAe}.
 Now it follows, from~\eqref{new2} and Lemma~\ref{eAe}, that $S(B^+,n,r)$ is of infinite type for any $n\geq 3$ and $r\geq 2$.

 Finally, the algebra $S(B^+,n,1)$ is isomorphic to $K\cQ$ where $\cQ$ is a quiver of type $A_n$ with linear orientation, hence it has finite type.
\end{proof}

        \bibliography{almostsplit}
\bibliographystyle{amsplain}

\end{document}